\documentclass[a4paper,11pt,oneside,fleqn]{article}

\usepackage{amsmath}
\usepackage{amssymb}
\usepackage{amsthm}
\newtheorem{thm}{Theorem}
\newtheorem{prop}[thm]{Proposition}
\newtheorem{lem}[thm]{Lemma}
\newtheorem{cor}[thm]{Corollary}
\newtheorem{cnj}[thm]{Conjecture}

\theoremstyle{remark}
\newtheorem{rmk}[thm]{Remark}

\usepackage{mathtools}
\usepackage{multirow}
 \newcommand{\Nat}{\mathbb{N}}
 \newcommand{\Integer}{\mathbb{Z}}
 \newcommand{\Real}{\mathbb{R}}

 \newcommand\xqed[1]{%
  \leavevmode\unskip\penalty9999 \hbox{}\nobreak\hfill
  \quad\hbox{#1}}
 \newcommand\qei{\xqed{$\triangle$}}
 \DeclareMathOperator{\Fix}{Fix}
 \DeclareMathOperator{\Id}{Id}
 \DeclareMathOperator{\Hf}{H}
 \DeclareMathOperator{\PF}{P}
 \DeclareRobustCommand{\VAN}[3]{#2}

\usepackage{enumitem}
\usepackage{graphicx,setspace}
\usepackage{epstopdf}
\usepackage{float}
\usepackage[section]{placeins}
\usepackage[bf,small,textfont=rm]{caption}

\usepackage{harvard}
\citationstyle{dcu}
\usepackage{hyperref}
\hypersetup{
pdfborder={0 0 1},
hypertexnames=false,
breaklinks=true,
colorlinks=true,
linkcolor=blue,
urlcolor=red,
citecolor=cyan,
pdftitle={Symmetries in the Lorenz-96 model},
pdfauthor={Dirk L. van Kekem and Alef E. Sterk},
pdfsubject={Analytical study of the Lorenz-96 model reveals its symmetries and proves the existence of a pitchfork bifurcation in each even dimension and a second one for dimensions that are multiples of four. Numerically it is shown that a full structure of \emph{q} pitchfork bifurcations exists in each dimension \emph{n}, where \emph{q} is a nonnegative integer denoting the powers of 2 in \emph{n}.},
pdfkeywords={Lorenz-96 model, Dynamical systems, Symmetry, Equivariant dynamical system, Pitchfork bifurcation, Center manifold, Invariant manifold}
}

\makeatletter
\gdef\harvardcite#1#2#3#4{%
  \global\@namedef{HAR@fn@#1}{\hyper@@link[cite]{}{cite.#1}{#2}}%
  \global\@namedef{HAR@an@#1}{\hyper@@link[cite]{}{cite.#1}{#3}}%
  \global\@namedef{HAR@yr@#1}{\hyper@@link[cite]{}{cite.#1}{#4}}%
  \global\@namedef{HAR@df@#1}{\csname HAR@fn@#1\endcsname}%
}
\makeatother

\author{Dirk L. van Kekem \& Alef E. Sterk}
\title{Symmetries in the Lorenz-96 model}
\date{\today}

\begin{document}
\maketitle

\begin{abstract}
\noindent The Lorenz-96 model is widely used as a test model for various applications, such as data assimilation methods. This symmetric model has the forcing $F\in\Real$ and the dimension $n\in\Nat$ as parameters and is $\Integer_n$-equivariant. In this paper, we unravel its dynamics for $F<0$ using equivariant bifurcation theory. Symmetry gives rise to invariant subspaces, that play an important role in this model. We exploit them in order to generalise results from a low dimension to all multiples of that dimension. We discuss symmetry for periodic orbits as well.

Our analysis leads to proofs of the existence of pitchfork bifurcations for $F<0$ in specific dimensions $n$: In all even dimensions, the equilibrium $(F,\ldots,F)$ exhibits a supercritical pitchfork bifurcation. In dimensions $n=4k$, $k\in\Nat$, a second supercritical pitchfork bifurcation occurs simultaneously for both equilibria originating from the previous one.

Furthermore, numerical observations reveal that in dimension $n=2^qp$, where $q\in\Nat\cup\{0\}$ and $p$ is odd, there is a finite cascade of exactly $q$ subsequent pitchfork bifurcations, whose bifurcation values are independent of $n$. This structure is discussed and interpreted in light of the symmetries of the model.
\end{abstract}

\section{Introduction}\label{sec:Intro}
\subsection{Setting of the problem}
\paragraph{Equivariant dynamical systems}It was not until the late 1970s that the study of symmetric dynamical systems gained great interest, when it was discovered that the symmetries of a system can have a big impact on its dynamics. Since then, a considerable amount of literature has been published on symmetry and bifurcations resulting in a rich extension of bifurcation theory, called \emph{equivariant bifurcation theory}. In this field a group-theoretic formalism is used to classify bifurcations and to describe solutions and other phenomena of a system. One of its most powerful results of this so-called equivariant bifurcation theory is the equivariant branching lemma, formulated first by \citeasnoun{Vanderbauwhede82} and \citeasnoun{Cicogna81} independently. A few years later a detailed overview of the theory of local equivariant bifurcations appeared in \cite{Golubitsky85,Golubitsky88}, which is still a standard reference in this field. This is followed in more recent years by other works with an overview of the new state-of-the-art, e.g.~\cite{Chossat00} with an applied mathematics approach and the more advanced and theoretical work \cite{Field07}.

Knowing the symmetries of a model can provide a lot of insight in the dynamics via equivariant bifurcation theory. This includes the occurrence of certain bifurcations, symmetry-related solutions, pattern formation and invariant manifolds (see for many concrete examples \cite{Chossat00,Golubitsky88} and references therein). Also, there are many examples of phenomena in nature that have some symmetry. To illustrate: the Rayleigh-B\'enard convection possesses a reflection symmetry, which leads to a pitchfork bifurcation (among others), as can be concluded by the equivariant branching lemma \cite{Golubitsky84}. Likewise, the Lorenz-63 model exhibits a pitchfork bifurcation due to symmetry, as it is derived from the Rayleigh-B\'enard convection \cite{Lorenz63}.

\paragraph{Lorenz-96 model}This paper concentrates on another model of Lorenz, namely, his 1996 model \cite{Lorenz06}. Already in 1984 he studied a 4-dimensional version of this model in his search for the simplest nontrivial forced dissipative system that is capable of exhibiting chaotic behaviour \cite{Lorenz84a}. By imposing symmetry conditions on the equations, he introduced the monoscale version of the $n$-dimensional Lorenz-96 model. The equations of this model are equivariant with respect to a cyclic permutation of the variables and so the system is completely determined by the equation for the $j$-th variable, which is given by
\begin{subequations}\label{eq:Lorenz96}
\begin{equation}
    \dot{x}_j =  x_{j-1}(x_{j+1}-x_{j-2}) - x_j + F, \qquad j = 1, \ldots, n,\label{eq:Lorenz96eq}
\end{equation}
and a `boundary condition'
\begin{equation}
    x_{j-n} = x_{j+n} = x_j.\label{eq:Lorenz96bc}
\end{equation}
\end{subequations}
Here, the variables $x_j$ can be associated to values of some atmospheric quantity (e.g.\ temperature) measured along a circle of constant latitude of the earth \cite{Lorenz06}. The latitude circle is divided into $n$ equal parts such that the index $j=1,\ldots,n$ denotes the longitude of a particular variable. The number $n\in\Nat$ is the dimension of the system, while the forcing $F\in\Real$ can be used as a bifurcation parameter.

The Lorenz-96 model is used by Lorenz to study the atmosphere and related problems \cite{Lorenz98,Lorenz06,Lorenz06b}. The simplicity of the model makes it also attractive and useful for various other applications, such as to test data assimilation methods \cite{Leeuw17-1,Ott04,Trevisan11} and to study spatiotemporal chaos \cite{Pazo08}. For a more complete overview of studies that exploit the Lorenz-96 model for applications, we refer to our papers \cite{Kekem17a1,Kekem17c1}.

Table~\ref{tab:DynamicsLz96} lists a selection of papers that investigate part of the dynamics of the Lorenz-96 model. In a recent article, we have proven analytically some basic properties for all dimensions (but mainly for positive parameter values) and the existence of Hopf and Hopf-Hopf bifurcations and we have studied numerically the routes to chaos for $F>0$ \cite{Kekem17a1}. However, there are only two papers that focus on the dynamics for negative parameter values $F$: in \cite{Lorenz84a} the chaotic attractor is studied for $F=-100$. In \cite{Kekem17c1} we have investigated the spatiotemporal properties of waves for both $F>0$ and $F<0$ and showed mostly numerically that the dynamics for $F<0$ depends on the parity of the dimension. A systematic understanding of how the symmetry influences the dynamics of the Lorenz-96 model is however still lacking. Therefore, we continue in this paper by using an analytical approach to examine the nature of the symmetry of the Lorenz-96 model and their implications for bifurcation sequences. We mainly focus on negative parameter values $F$, since for those values the symmetry has a larger influence on the dynamics, due to the existence of pitchfork bifurcations that are induced by symmetry as well. Whenever relevant, we will also discuss the implications of our findings on symmetry for the case of positive $F$.

\begin{table}[ht!]
\caption{Overview of the research into the dynamics of the monoscale Lorenz-96 model~\eqref{eq:Lorenz96} and the main values of the parameter $F$ that were used. In most cases, only the range for positive $F$ has been analysed.}
\makebox[\textwidth][c]{
\begin{tabular}{llr}
\hline
Reference & Subject & $F$ \\
\hline
\citeasnoun{Lorenz84}      & Chaotic attractor                      & $-100$ \\
\citeasnoun{Orrell03}      & Spectral bifurcation diagram           & $[0,17)$ \\
\citeasnoun{Lorenz05}      & Designing chaotic models               & $2.5, 5, 10, 20, 40$\\
\citeasnoun{Pazo08}        & Lyapunov vectors                       & $8$ \\
\citeasnoun{Karimi10}      & Extensive chaos                        & $[5,30]$ \\
\citeasnoun{Kekem17a1}      & Travelling waves \& bifurcations       & $[0,13)$ \\
\citeasnoun{Kekem17c1}       & Wave propagation                       & $(-4,4)$ \\
\hline
\end{tabular}
}
\end{table}\label{tab:DynamicsLz96}

\subsection{Summary of the results}
The main results of this paper can be summarised as follows: first of all, for any $n\in\Nat$ the $n$-dimensional Lorenz-96 model is equivariant with respect to a cyclic left shift of the coordinates (denoted by $\gamma_n$), i.e.\ the model has a $\Integer_n$-symmetry. It is well-known that equivariance gives rise to invariant linear subspaces. These invariant subspaces turn out to have very important implications for the dynamics of the model. We show how they can be utilised in extrapolating established facts in a certain dimension to all multiples of that dimension and exploit them to clarify the dynamical structure.

The trivial equilibrium $x_F = (F,\ldots, F)$, which exists for all $n\geq 1$ and all $F\in\Real$, is invariant under $\gamma_n$. Of particular interest is when the dimension $n$ is even, in which case $\Integer_2$-symmetry can be realised by $\gamma_n^{n/2}$. Equivariant bifurcation theory then shows that the equilibrium $x_F$ exhibits a pitchfork bifurcation. The emerging stable equilibria both exhibit again a pitchfork bifurcation if $n$ is a multiple of four. Both cases will be proven for the smallest possible dimension (i.e.~$n=2$, resp.~$n=4$) using a theorem from \citeasnoun{Kuznetsov04} on bifurcations for systems with $\Integer_2$-symmetry. A generalisation to all dimensions $n=2k$, resp.~$n=4k$, is then provided by the invariant manifolds.

Furthermore, a supercritical Hopf bifurcation destabilises all present stable equilibria after at most two pitchfork bifurcations, as is shown numerically in \cite{Kekem17c1}. Therefore, the dynamical structure for $n\geq4$ and $F<0$ can be divided in general into three classes, depending on the dimension $n$ (see also Figures~\ref{fig:Lz96-PFStructureFigOdd}--\ref{fig:Lz96-PFStructureFig4m} for schematic bifurcation scenarios):
\begin{enumerate}
  \item If $n$ is odd, then the first bifurcation of the equilibrium $x_F$ is a supercritical Hopf bifurcation.
  \item If $n=4k+2$, $k\in\Nat$, then only one pitchfork bifurcation takes place, followed by a Hopf bifurcation on each branch. This leads to two stable periodic orbits that coexists for the same parameter values~$F$.
  \item If $n=4k$, $k\in\Nat$, then all four stable equilibria generated by the second pitchfork bifurcation exhibit a Hopf bifurcation simultaneously, resulting in four coexisting stable periodic orbits.
\end{enumerate}

Of particular interest is the observation that in the last case there can be even more pitchfork bifurcations, that however occur after the equilibria undergo the Hopf bifurcations. We conjecture that the number of subsequent pitchfork bifurcations depends on the dimension $n$ as follows: let $n=2^qp$, where $q\in\Nat\cup\{0\}$ and $p$ is odd, then the number of successive pitchfork bifurcations is exactly equal to $q$. This finite cascade of $q$ pitchfork bifurcations leads to a structure of $2^{q+1}-1$ equilibria that are mutually conjugate by a power of $\gamma_n$. An example of such a structure is given by the schematic bifurcation diagram in Figure~\ref{fig:Lz96-PFStructureFullFig} for $n=2^4p$.

For positive forcing the first bifurcation is a Hopf or Hopf-Hopf bifurcation, which is not induced by symmetry. However, it turns out that the generated periodic orbits can be symmetric if their wave number has a common divisor with the dimension. Here, the wave number should be interpreted as the spatial frequency of the wave, which measures the number of `highs' or `lows' on the latitude circle, see for example \cite{Kekem17a1,Kekem17c1,Lorenz98}.

\subsection{Overview}
This paper has been divided into two parts. The first part, section~\ref{sec:AnalyticalResults}, deals with the analytical results of the research. We give an exposition of the symmetries of the Lorenz-96 model and corresponding invariant manifolds, using concepts from equivariant bifurcation theory. These results are used to prove the existence of a pitchfork bifurcation in all even dimensions and a second occurrence of pitchfork bifurcations in all dimensions of the form $n=4k$.

The second part of this paper, section~\ref{sec:NumericalResults}, is devoted to a further, numerical exploration of the dynamics for all dimensions. First, we will see that the periodic orbits after a supercritical Hopf bifurcation still have some symmetry. Secondly, we investigate numerically and explain by our exposition of symmetry the existence of the structure with exactly $q$ subsequent pitchfork bifurcations in dimension $n=2^qp$.

\section{Analytical results}\label{sec:AnalyticalResults}
In this section we will describe the symmetry of the Lorenz-96 model using concepts from equivariant dynamical systems theory. For a detailed and clear overview of this field, we refer to the standard textbooks on bifurcation theory for symmetric systems \cite{Golubitsky85,Golubitsky88}. The symmetry also gives rise to invariant manifolds, that turn out to be very important in our model. We will show how these invariant manifolds can be exploited. Lastly, we show that the symmetry of the model leads to subsequent pitchfork bifurcations for particular dimensions.

\subsection{$\Integer_n$-symmetry and invariant manifolds}\label{sec:symmetries}
\paragraph{Cyclic symmetry}Let $n\in\Nat$ be arbitrary and denote the right-hand-side of system~\eqref{eq:Lorenz96eq} with dimension $n$ by $f_n(x,F)$, such that $f_n:\Real^n\times\Real \rightarrow\Real^n$. Consider the following $n$-dimensional permutation matrix:
\begin{equation}\label{eq:generator}
  \gamma_n = \begin{pmatrix}
             0 & 1 & 0 & \cdots & 0 \\
              & 0 & 1 &  & \\
             \vdots &  & \ddots & \ddots & \vdots \\
             0 &  & \cdots & 0 & 1 \\
             1 & 0 & \cdots &  & 0
           \end{pmatrix}.
\end{equation}
It is obvious that the linear mapping $\gamma_n:\Real^n\rightarrow\Real^n$ acts like a cyclic left shift and that $\gamma_n^n = \Id_n$, the $n$-dimensional identity matrix. We define the cyclic group generated by $\gamma_n$ as
\begin{equation*}
    \Gamma_n :=  \left\langle\gamma_n\right\rangle,
\end{equation*}
which is isomorphic to the additive group $\mathbb{Z}/n\mathbb{Z}$. A key observation is that
\begin{equation*}
    f_n(\gamma_n^j x,F) = \gamma_n^j f_n(x,F)
\end{equation*}
holds for any $j \in \Nat$ and any $n\in \Nat$. This immediately implies the following result:
\begin{prop}[$\Integer_n$-symmetry]\label{prop:L96equivariant}
  For any dimension $n\geq1$ the Lorenz-96 model is $\Gamma_n$-equivariant. \qed
\end{prop}

Furthermore, powers of $\gamma_n$ generate subgroups of $\Gamma_n$, namely,
\begin{equation}\label{eq:subgroup}
  G_n^m = \left\langle\gamma_n^m\right\rangle < \Gamma_n, \qquad 0 < m \leq n, \quad m | n.
\end{equation}
The order of a subgroup $G_n^m$ is $n/m$ and it is isomorphic to $\Integer/(n/m)\Integer$. These subgroups $G_n^m$ are \emph{isotropy subgroups} of special equilibrium solutions that are of the form
\begin{equation}\label{eq:solutionxm}
    x^m = (A_m,\ldots, A_m), \qquad A_m = (a_0, \ldots, a_{m-1}),\qquad 0 < m \leq n,\quad m|n,
\end{equation}
where $a_j \in \Real$. The coordinates of $x^m$ have $n/m$ repetitions of the block $A_m$ and, hence, $\gamma_n^{km} x^m = x^m$ for all $k\in\Nat$. Later on, we will encounter equilibria of system~\eqref{eq:Lorenz96} which have indeed such a structure (see sections~\ref{sec:PF1} and~\ref{sec:PFcascade}).

\paragraph{Invariant manifolds}Associated to an isotropy subgroup $G < \Gamma_n$ is the \emph{fixed-point subspace} $\Fix(G)$, i.e.\ an invariant linear subspace consisting of all points in $\Real^n$ that satisfy $\gamma x = x$ for any element $\gamma\in G$. It is a well-known result that such a fixed-point subspace is an invariant set of the dynamical system \cite{Golubitsky88}. Here, the fixed-point subspace which is fixed by the complete subgroup $G_n^m$ is given by
\begin{equation}\label{eq:fixedpointsubspaceGmn}
  \Fix(G_n^m) = \{x\in\Real^n: x = x^m\},
\end{equation}
where $x^m$ is as in equation~\eqref{eq:solutionxm}.

The fixed-point subspace~\eqref{eq:fixedpointsubspaceGmn} is an invariant manifold of dimension $m$.  Also, each invariant manifold of dimension $m|n$ contains all of its `predecessors' with dimension $m'$ such that $m'|m$:
\begin{equation*}
  \Fix(\Gamma_n) = \Fix(G^{1}_n) \subset \Fix(G^{m'}_n) \subset \Fix(G^m_n) \subset \Fix(G^{n}_n) = \Real^n.
\end{equation*}
These invariant subspaces constitute nested families of subspaces which are all invariant under the flow of the Lorenz-96 model.

Equivariance also implies that if $x(t)$ is a solution of the Lorenz-96 model, then its \emph{conjugate} solutions, $\gamma_n^j x(t)$, are solutions as well for any $j\in\Nat$. Moreover, for each equilibrium solution $x^m\in\Fix(G^m_n)$ all conjugate solutions $\gamma_n^j x^m$ with $0\leq j < m$ are elements of the same fixed-point subspace $\Fix(G^m_n)$ and have the same properties\footnote{For example, the eigenvalues of conjugate solutions are equal.} (up to symmetry) as $x^m$ by permutation of the governing equations. This fact allows us to study only one of the equilibria in the \emph{group orbit} $\{x: x=\gamma_n^j x^m, 0\leq j < m \}$ of $x^m$.

\paragraph{Reduction of dimension}The preceding observation, together with the fact that the coordinates of points $x^m$ have repetitions when $m<n$, simplifies our analysis a lot. In particular, it implies that we can reduce the number of governing equations of the system inside $\Fix(G_n^m)$. In fact, we have the following important result:
\begin{prop}\label{prop:invmanifold}
    Let $m\in\Nat$ and let $n= k m$ be any multiple of $m$. The dynamics of the $n$-dimensional Lorenz-96 model restricted to the invariant manifold $\Fix(G_{n}^m)$ is topologically equivalent to the Lorenz-96 model of dimension $m$.
\end{prop}
\begin{proof}
  Let $n\in\Nat$ be as given and restrict the $n$-dimensional Lorenz-96 model to the invariant manifold $\Fix(G^m_{n})$. By definition~\eqref{eq:fixedpointsubspaceGmn} we have that the entries of any $x\in \Fix(G^m_{n})$ repeat as $x_{j+m} = x_j$ with the index modulo $n$. It follows immediately that equation~\eqref{eq:Lorenz96eq} for the $(j+m)$-th coordinate equals that for the $j$-th coordinate. Hence, we are left with $n/m$ copies of an $m$-dimensional Lorenz-96 model on $\Fix(G^m_{n})$.

  Furthermore, since $\Fix(G^m_{n})$ and each of its copies have dimension $m$, the dynamics on $\Fix(G^m_{n})$ is governed by $m$ equations only and hence by the Lorenz-96 model of dimension $m$. Hence, on $\Fix(G^m_{n})$ we can reduce to a lower dimensional model. As homeomorphism between the invariant manifold $\Fix(G^m_{n})$ and $\Real^{m}$ (the space of the $m$-dimensional Lorenz-96 system) one can take the function which selects the first $m$ coordinates and drops the remaining coordinates, leaving us with the $m$-dimensional Lorenz-96 model. Its inverse is then the map which duplicates the given $m$ coordinates $n/m$ times.
\end{proof}

\begin{rmk}\label{rmk:invmanifold}
  Proposition~\ref{prop:invmanifold} enables us to generalise results from low dimensions to higher dimensions. For example, when in the $m$-dimension\-al Lorenz-96 model a certain bifurcation occurs, then generically for every multiple $n= k m$, $k\in\Nat$, the same bifurcation occurs in the $n$-dimensional model. This vastly reduces the proof of facts that occur in many dimensions, since it comes down to search for the lowest possible dimension to occur and to prove it for that particular dimension. By Proposition~\ref{prop:invmanifold} then, this proves the property for infinitely many dimensions.

  A note of caution is due here, since two problems can occur:
  \begin{enumerate}
    \item It might happen that another bifurcation will take place before the phenomena extrapolated from a lower dimension and thus a different attractor gains stability, resulting in a different route to chaos.
    \item Besides that, another attractor can exist with no or different symmetry (i.e.\ in another subspace than $\Fix(G_n^m)$) and whose route to chaos is different.
  \end{enumerate}
  In both cases, chaos possibly occurs for smaller parameter values. What the method of Proposition~\ref{prop:invmanifold} does provide, are the features and bifurcations of the attractors inside the subspace $\Fix(G_n^m)$, for any $n$ that is a multiple of $m$. \qei
\end{rmk}

\subsection{First pitchfork bifurcation}\label{sec:PF1}
System~\eqref{eq:Lorenz96} has in any dimension the trivial equilibrium
\begin{equation}\label{eq:Lztriveq}
  x_F = (F,\ldots,F), \quad F \in \Real.
\end{equation}
The eigenvalues of this equilibrium can be determined easily using the circulant nature of the Jacobian matrix. Let
\begin{equation*}
\rho_j = \exp\left(-2\pi i \tfrac{j}{n}\right),
\end{equation*}
then it is shown in \cite{Kekem17a1} that these eigenvalues are given by
\begin{equation}\label{eq:Lzsimpleevn}
  \lambda_j(F,n)
    = -1 + F\rho_j^1 - F\rho_j^{n-2}
\end{equation}
for all $j=0,\dots,n-1$. We omit the dependence on $n$ from now on and write $\lambda_j$ or $\lambda_j(F)$ for $\lambda_j(F,n)$. The eigenvector corresponding to $\lambda_j$ can be expressed in terms of $\rho_j$ as well:
\begin{equation}\label{eq:Lzeigenvector}
v_j = \frac{1}{\sqrt{n}}\begin{pmatrix*} 1 & \rho_j & \rho_j^2 & \cdots & \rho_j^{n-1} \end{pmatrix*}^{\top}.
\end{equation}

Observe that the eigenvalue $\lambda_0$ equals $-1$. Due to the fact that $\rho_{n-j} = \bar{\rho}_j$, all the other eigenvalues and eigenvectors form conjugate pairs as
\begin{align}\label{eq:Lzevconj}
  \lambda_j &= \bar{\lambda}_{n-j},\\
  v_j &= \bar{v}_{n-j},\nonumber
\end{align}
except when $n$ is even, in which case the eigenvalue for $j = \tfrac{n}{2}$ is real and equals $\lambda_{n/2} = -1 - 2F$. This is the only eigenvalue that depends on the parameter $F$ and is purely real and thus plays a key role in this study. We can have more real eigenvalues when $n$ is a multiple of 3, in which case the eigenvalues for $j = \tfrac{n}{3}, \tfrac{2n}{3}$ are both fixed and equal to $-1$.

For every even dimension $n$ the eigenvalue $\lambda_{n/2}$ equals $0$ at $F = -\tfrac{1}{2}$. This gives rise to the first of our main results:
\begin{thm}[First pitchfork bifurcation]\label{thm:PFBif}
  Let $n \in \Nat$ be even. Then the trivial equilibrium $x_F$ exhibits a supercritical pitchfork bifurcation at the parameter value $F_{\PF,1} := -\tfrac{1}{2}$.
\end{thm}
Note that the index $1$ of $F_{\PF}$ anticipates the possibility of more pitchfork bifurcations, of which this is the first one in line for decreasing $F$. We prove Theorem~\ref{thm:PFBif} using the symmetries of the model. By Proposition~\ref{prop:invmanifold}, proving the theorem boils down to proving the existence of a pitchfork bifurcation in the case $n=2$  and subsequently generalise it to any even dimension $n$. The same procedure will be applied to prove the existence of a second occurrence of pitchfork bifurcations in section~\ref{sec:PF2}.

The proof of Theorem~\ref{thm:PFBif} (and also that of the second pitchfork bifurcation) rely on a theorem taken from \cite{Kuznetsov04}. Before we state this result, let us first introduce some notation. Let $R_n$ be an $n\times n$ matrix that defines a symmetry transformation $x \mapsto R_n x$. Furthermore, we decompose $\Real^n$ into a direct sum $\Real^n = X^+_n \oplus X^-_n$, where
\begin{align*}
    X^+_n &:= \{x\in\Real^n: R_n x = x \}, \\
    X^-_n &:= \{x\in\Real^n: R_n x = -x \}.
\end{align*}

\begin{thm}[Kuznetsov, 2004]\label{thm:Kuznetsov}
  Suppose that a $\Integer_2$-equivariant system
  \begin{equation*}
    \dot{x} = f(x, \alpha), \quad x \in \Real^n,\quad \alpha \in \Real^1,
  \end{equation*}
  with smooth $f$, $R_n f(x, \alpha) = f(R_n x, \alpha)$ and $R_n^2 = \Id_n$, has at $\alpha = 0$ the fixed equilibrium $x_0 = 0$ with simple zero eigenvalue $\lambda_1 = 0$, and let $v \in \Real^n$ be the corresponding eigenvector.

  Then the system has a 1-dimensional $R_n$-invariant center manifold $W^c_\alpha$ and one of the following alternatives generically takes place:
  \begin{enumerate}[label=\emph{(\roman*)}]
    \item \emph{(fold)} If $v \in X^+_n$, then $W^c_\alpha \subset X^+_n$ for all sufficiently small $|\alpha|$ and the restriction of the system to $W^c_\alpha$ is locally topologically equivalent near the origin to the normal form
        \begin{equation*}
            \dot{\xi} = \beta \pm \xi^2;
        \end{equation*}
    \item \emph{(pitchfork)} If $v \in X^-_n$, then $W^c_\alpha \cap X^+_n = x_0$ for all sufficiently small $|\alpha|$ and the restriction of the system to $W^c_\alpha$ is locally topologically equivalent near the origin to the normal form
        \begin{equation*}
            \dot{\xi} = \beta\xi \pm \xi^3. \hfill \qed
        \end{equation*}
  \end{enumerate}
\end{thm}

\begin{rmk}\label{rmk:Kuznetsov}
  At the pitchfork bifurcation the equilibrium that satisfies $R_n x_0 = x_0$ changes stability, while two $R_n$-conjugate equilibria appear \cite{Kuznetsov04}. In terms of the fixed-point subspaces, this means that the resulting $R_n$-conjugate equilibria are contained in a larger subspace than the original. In section~\ref{sec:PFcascade} we will elaborate further on this.

  The proofs below of the first, resp.\ the second (in the next section), pitchfork bifurcation are based on the lowest possible dimensions, i.e.\ $m=2$ and $m=4$. In both cases we start with equilibria in $\Fix(G_m^{m/2})$ and $\Integer_2$-symmetry is realised by $\gamma_m^{m/2}$. Consequently, we will set
  \begin{equation}\label{eq:R_n}
    R_m:=\gamma_m^{m/2},
  \end{equation}
  and the pitchfork bifurcation will result in two extra $\gamma_m^{m/2}$-conjugate equilibria in $\Fix(G_m^m)$. Likewise, we have that $X^+_m = \Fix(G_m^{m/2})$ and $X^-_m = \Fix(G_m^{m/2})^{\bot}$.

  For general dimensions $n = km$ we can extend these results according to Proposition~\ref{prop:invmanifold} which yields that the equilibria after the first pitchfork bifurcation (for which $m=2$) are $\gamma_{km}^{m/2} = \gamma_n^1$-conjugate and contained in $\Fix(G_n^2)$. Similarly, for the second pitchfork bifurcation we have $m=4$, so here the resulting equilibria will be pairwise $\gamma_{km}^{m/2}=\gamma_n^2$-conjugate and contained in $\Fix(G_n^4)$. \qei
\end{rmk}

In order to prove the existence of a pitchfork bifurcation in the 2-dimensional Lorenz-96 model, it suffices to show that it satisfies the second case of Theorem~\ref{thm:Kuznetsov}. This is in brief how the following lemma is proven:
\begin{lem}\label{lem:PF1n2}
  Let $n = 2$, then the equilibrium $x_F$ of the Lorenz-96 model exhibits a pitchfork bifurcation at the parameter value $F_{\PF,1}=-\tfrac{1}{2}$.
\end{lem}
\begin{proof}
  The eigenvalues of $x_F$ are given by equation~\eqref{eq:Lzsimpleevn}, so that in the 2-dimensional case we have
  \begin{equation*}
    \lambda_0 = -1, \quad \lambda_1 = -1 - 2F.
  \end{equation*}
  Therefore, $\lambda_1 = 0$ at $F=F_{\PF,1}$ and a bifurcation takes place. An eigenvector at $F_{\PF,1}$ corresponding to $\lambda_1$ is given by
  \begin{equation*}
    v_1 = (-1, 1).
  \end{equation*}

  By Proposition~\ref{prop:L96equivariant} system~\eqref{eq:Lorenz96} with $n=2$ has a $\Integer_2$-equivariance with symmetry transformation
  \begin{equation}\label{eq:equivmatrix2}
    R_2 := \gamma_2 = \begin{pmatrix}
             0 & 1 \\
             1 & 0
           \end{pmatrix},
  \end{equation}
  as defined by formula~\eqref{eq:R_n}. From section~\ref{sec:symmetries} it follows that this matrix satisfies the following:
  \begin{enumerate}
    \item $R_2^2 = \Id_2$;
    \item $R_2$ defines a symmetry transformation on $\Real^2 = X^+_2 \oplus X^-_2$ with
        \begin{align*}
          X^+_2 &= \Fix(G_2^1),\\
          X^-_2 &= \Fix(G_2^1)^{\bot} = \{x\in\Real^2: x_0 = -x_1\}.
        \end{align*}
  \end{enumerate}
  With these preliminaries the conditions of Theorem~\ref{thm:Kuznetsov} are satisfied (up to a transformation to the origin). In addition, it is easy to see that we are in the pitchfork-case, since we have
  \begin{equation*}
    R_2 v_1 = -v_1,
  \end{equation*}
  i.e.\ the eigenvector with respect to $\lambda_1(F_{\PF,1})$ lies in $X^-_2$. Hence, by Theorem~\ref{thm:Kuznetsov} the 2-dimensional Lorenz-96 model has a 1-dimensional $R_2$-invariant center manifold $W^c_F$ with $W^c_F \cap X^+_2 = x_F$ for all $F$ sufficiently close to $F_{\PF,1}$ and the restriction of the system to $W^c_F$ is locally topologically equivalent near $x_F$ to the normal form of a pitchfork bifurcation.
\end{proof}

\begin{proof}[Proof of Theorem~\ref{thm:PFBif}]
  The result of Lemma~\ref{lem:PF1n2} extends to all dimensions $n = 2k, k \in \Nat$ by Proposition~\ref{prop:invmanifold}.
\end{proof}

\begin{rmk}
  Theorem~\ref{thm:PFBif} can also be proven via a center manifold reduction \cite{Guckenheimer83,Kuznetsov04,Wiggins03}, which gives the following form of system~\eqref{eq:Lorenz96eq} (up to linear transformations), restricted to its center manifold:
  \begin{equation}\label{eq:CMPF}
     \dot{u} = -2\alpha u-\frac{4}{n}u^3 + \mathcal{O}(\|u,\alpha\|^4),
  \end{equation}
  where $\alpha = F+\tfrac{1}{2}$. This is the normal form of the supercritical pitchfork bifurcation and implies that the equilibrium $x_F$ is stable for $F > F_{\PF,1}$ and loses stability at $F = F_{\PF,1}$, while two other stable equilibria exist for $F < F_{\PF,1}$. \qei
\end{rmk}

At the supercritical pitchfork bifurcation the equilibrium $x_F\in \Fix(G_n^1)$ loses stability and gives rise to two stable equilibria $\xi^1_j\in \Fix(G_n^2)$, $j=0,1$, that exist for $F<-\tfrac{1}{2}$. These new equilibria are given by
\begin{equation}\label{eq:eqxi1}
    \xi_0^1(F) = (a_+,a_-,\ldots,a_+,a_-), \quad a_{\pm} = -\tfrac{1}{2}\pm \tfrac{1}{2}\sqrt{-1-2F},
\end{equation}
while $\xi_1^1$ is obtained by swapping the indices $+$ and $-$. So, each $\xi_j^1$ has a structure like the equilibria $x^m$ in formula~\eqref{eq:solutionxm} with $m=2$ and they are indeed $\gamma_n$-conjugate as predicted by Remark~\ref{rmk:Kuznetsov}. In other words: applying the matrix $\gamma_n$ means geometrically a switch from one branch of equilibria to the other.

\subsection{Second pitchfork bifurcation}\label{sec:PF2}
The pitchfork bifurcation described in the previous section is followed by a second subsequent pitchfork bifurcation for $F<F_{\PF,1}$ when the dimension is a multiple of 4. This time, there are two simultaneous bifurcations, each of which takes place at a different branch of equilibria~\eqref{eq:eqxi1} that emanated from the first pitchfork bifurcation of Theorem~\ref{thm:PFBif}.

\begin{thm}[Second pitchfork bifurcation]\label{thm:PF2Bif}
  Let $n = 4k$ with $k \in \Nat$. Then both equilibria $\xi^1_{0,1}(F)$ emanating from the pitchfork bifurcation at $F_{\PF,1} = -\tfrac{1}{2}$ exhibit a supercritical pitchfork bifurcation at the parameter value $F_{\PF,2}:=-3$.
\end{thm}
The proof goes in exactly the same way as the proof for the first pitchfork bifurcation. Again, we first prove a lemma that describes the occurrence of a second pitchfork bifurcation in the lowest possible dimension:
\begin{lem}\label{lem:PF2n4}
  Let $n = 4$, then the equilibria $\xi^1_{0,1}(F)$ emanating from the pitchfork bifurcation at $F_{\PF,1} = -\tfrac{1}{2}$ both exhibit a pitchfork bifurcation at the parameter value $F_{\PF,2} := -3$.
\end{lem}
\begin{proof}
  The eigenvalues of both equilibria $\xi^1_j$, $j=0,1$ are given by:
  \begin{align}\label{eq:evxi1}
    \lambda^1_{0,1} &= \tfrac{1}{2}(-1 \pm \sqrt{9 + 16 F}),\nonumber\\
    \lambda^1_{2,3} &= \tfrac{1}{2}(-3 \pm \sqrt{-3 - 4 F}).
  \end{align}
  Since $\lambda_2^1=0$ when $F=F_{\PF,2}$, a bifurcation takes place. An eigenvector corresponding to $\lambda_2^1(F_{\PF,2})$ is given by
  \begin{equation*}
    v^1_2 = (2 + \sqrt{5}, -1, -2 - \sqrt{5}, 1).
  \end{equation*}

  To show that system~\eqref{eq:Lorenz96} with $n=4$ is $\Integer_2$-equivariant we recall that
  \begin{equation}\label{eq:equivmatrix}
    R_4 := \gamma_4^2 = \begin{pmatrix}
             0 & 0 & 1 & 0 \\
             0 & 0 & 0 & 1 \\
             1 & 0 & 0 & 0 \\
             0 & 1 & 0 & 0
           \end{pmatrix},
  \end{equation}
  as defined by formula~\eqref{eq:R_n}. Again, this matrix satisfies the requirements of Theorem~\ref{thm:Kuznetsov}, since
  \begin{enumerate}
    \item $R_4^2 = \Id_4$;
    \item $R_4 f_4(x,F) = f_4(R_4 x,F)$, by Proposition~\ref{prop:L96equivariant};
    \item $R_4$ defines a symmetry transformation on $\Real^4 = X^+_4 \oplus X^-_4$, where
        \begin{align*}
          X^+_4 &= \Fix(G_4^2),\\
          X^-_4 &= \Fix(G_4^2)^{\bot} = \{x\in\Real^4: x_0 = -x_2, x_1 = -x_3\}.
        \end{align*}
  \end{enumerate}
  Note that the group $\{\Id_4,R_4\}$ has $\Fix(G_n^2)$ as its fixed-point subspace, so it contains all the symmetries of $\xi^1\in \Fix(G_n^2)$ (i.e.\ $R_4 \xi^1 = \xi^1$). In contrast, it holds that
  \begin{equation*}
    R_4 v^1_2 = -v^1_2,
  \end{equation*}
  i.e.\ the eigenvector with respect to $\lambda^1_2(F_{\PF,2})$ lies in $X^-_4$. By Theorem~\ref{thm:Kuznetsov} this implies that a pitchfork bifurcation takes place and the 4-dimensional Lorenz-96 model has a 1-dimensional $R_4$-invariant center manifold $W^c_F$ with $W^c_F \cap X^+_4 = \xi^1_j$ for all $F$ sufficiently close to $F_{\PF,2}$.
\end{proof}
\begin{rmk}
  Lemma~\ref{lem:PF2n4} can be proven by a center manifold reduction, like Theorem~\ref{thm:PFBif} for dimensions $n=2k$. For $n=4$, this gives
  the following normal form of a pitchfork bifurcation:
  \begin{equation*}
    \dot{u} = a(\alpha) u + b(\alpha) u^3,
  \end{equation*}
  with
  \begin{align*}
    \hspace{-2.5em} a(\alpha) &= \frac{\alpha (18\sqrt{5}\sqrt{5 - 2\alpha} + \alpha)}{54 (-5 + 2\alpha)},\\
    \hspace{-2.5em} b(\alpha) &= \frac{450 (145 + 61\sqrt{5}) + \alpha(\sqrt{5 - 2\alpha}(854 + 406 \sqrt{5}) - 180 (145 + 61\sqrt{5}))}{135 (23 + 3\sqrt{5}) (-5 + 2\alpha)}.
  \end{align*}
  where $\alpha = F - F_{\PF,2} = F + 3$. The function $b(\alpha)$ is negative for values of $\alpha$ around 0, hence both pitchfork bifurcations at $F_{\PF,2}$ for $n=4$ are supercritical.
  \qei
\end{rmk}

\begin{proof}[Proof of Theorem~\ref{thm:PF2Bif}]
  The result of Lemma~\ref{lem:PF2n4} extends to all dimensions $n=4k$, $k\in\Nat$ by Proposition~\ref{prop:invmanifold}.
\end{proof}

\begin{rmk}\label{rmk:dimn=4m-2}
  A generalisation to all $n=4k-2$ is not provided by Proposition~\ref{prop:invmanifold}. Indeed, the second pair of eigenvalues $\lambda^1_{2,3}$ of~\eqref{eq:eqxi1} occurs only in the form of equation~\eqref{eq:evxi1} when the dimension is of the form $n=4k$. If instead the dimension equals $n=2$ then there are no more eigenvalues that can cross the imaginary axis, whereas for $n=4k-2$, $k\geq 2$, the eigenvalue pairs are different from the case $n=4k$, as numerical computations show \cite{Kekem17c1}. Therefore, in dimensions $n=4k-2$, $k\in\Nat$, there will not be an additional pitchfork bifurcation, but the next bifurcation after the first pitchfork bifurcation will be a Hopf bifurcation, as we will see in section~\ref{sec:Hbif}.
  \qei
\end{rmk}

At the second pitchfork bifurcations the equilibria $\xi^1_j\in \Fix(G_n^2)$, with $j=0,1$, lose stability and four stable equilibria $\xi^2_j\in \Fix(G_n^4)$, $0\leq j\leq 3$ appear that exist for $F<-3$. In contrast with $\xi^1_j$ it is not feasible to derive analytic expressions for the equilibria $\xi^2_j$. By Remark~\ref{rmk:Kuznetsov}, we know that in the 4-dimensional case these new equilibria are pairwise $R_4$-conjugate in the following way: $\xi^2_j = R_4 \xi^2_{j+2}$ (with the index modulo 4); that is, the equilibria with index $j$ and $j+2$ emanate from the same equilibrium $\xi^1_j$ for $j=0,1$. By equivariance, the conjugate solutions $\gamma_4 \xi^2_j$ are equilibria as well for all $0\leq j \leq 3$. In fact, we observe numerically that this gives precisely the solutions from the other $R_4$-conjugate pair of solutions (see section~\ref{sec:PFcascade}), i.e.\ we can switch between all four equilibria by subsequently applying $\gamma_4$, e.g.~$\xi^2_j = \gamma_4^j \xi^2_0$.

For general dimensions $n=4k$ similar statements hold: the new equilibria satisfy $\xi^2_j = \gamma_n^2 \xi^2_{j+2}$ (with the index modulo $n$) and they are of the form~\eqref{eq:solutionxm} with $m=4$. This gives an extra argument why a symmetry breaking by a pitchfork bifurcation is not possible in dimensions $n=4k-2$: since $n$ is not divisible by 4, we cannot `fill' the $n$ coordinates of an equilibrium in $\Real^n$ completely by blocks of four and the invariant subspace $\Fix(G_n^4)$ does not exist.

\section{Numerical results}\label{sec:NumericalResults}
In specific dimensions of the Lorenz-96 model we observed more than two subsequent pitchfork bifurcations with a nice structure. In this section we will give a more detailed exposition on this symmetric dynamical structure. We mainly concentrate on the bifurcation pattern for $F<0$ by describing the codimension 1 bifurcations of the equilibria that are generated via one or more pitchfork bifurcations.

Firstly, we discuss the occurrence of a supercritical Hopf bifurcation for $F<0$, which is preceded by at most two pitchfork bifurcations. Since these results are already presented in \cite{Kekem17a1} and \cite{Kekem17c1} together with the spatiotemporal properties of the resulting periodic orbit, we focus here on their symmetrical properties. By analysing the dimension of their containing invariant subspace we can clarify the existence (and non-existence) of patterns in the dynamics.

Secondly, we show that in specific dimensions it is possible to have more subsequent pitchfork bifurcations after the Hopf bifurcation and the two proven pitchfork bifurcations. We discuss how many of them can be expected in each dimension. These consecutive pitchfork bifurcations then also generate a lot of unstable equilibria that may influence the dynamics.

Most of these results follow from numerical observations. We will interpret these observations by means of the theoretical exposition of the symmetry in section~\ref{sec:symmetries} without aiming to be complete. Especially, proving facts after many pitchfork bifurcations will become increasingly difficult, since the lowest dimension needed increases exponentially.

\subsection{Symmetric periodic orbits}\label{sec:Hbif}
\paragraph{Destabilising Hopf bifurcations}Recall that for $n=2$ only one pitchfork bifurcation is possible and no further bifurcation can happen. Moreover, in dimensions $n=1$ and 3 all eigenvalues are equal to $-1$, so that no bifurcation is possible at all. Apart from that, we show below that in any dimension $n\geq4$ the stable equilibria for negative parameter values $F$ eventually lose stability through a supercritical Hopf bifurcation and one or more stable periodic orbits will appear.

In \cite{Kekem17c1} we have shown that this will happen after at most two subsequent pitchfork bifurcations. Therefore, the bifurcation pattern can be divided into three different cases according to the number of pitchfork bifurcations that occur before the Hopf bifurcation:

\begin{description}
  \item[Case 1: no pitchfork bifurcations] For odd $n$, no pitchfork bifurcation will occur, but the first bifurcation of the trivial equilibrium~\eqref{eq:Lztriveq} for $F<0$ is a Hopf bifurcation at $F_{\Hf} (j,n) := 1/(\cos \tfrac{2\pi j}{n} - \cos \tfrac{4\pi j}{n})$ with $j=\tfrac{n-1}{2}$. In \cite{Kekem17a1}, we have proven that this first Hopf bifurcation is supercritical, which implies that the stable equilibrium $x_F$ loses stability and a stable periodic orbit appears after the bifurcation; see Figure~\ref{fig:Lz96-PFStructureFigOdd}.
  \item[Case 2: one pitchfork bifurcation] For $n=4k+2$, $k\in \Nat$, Remark~\ref{rmk:dimn=4m-2} states that only one pitchfork bifurcation occurs in this case, whose existence is proven by Theorem~\ref{thm:PFBif}. In \cite{Kekem17c1}, we have demonstrated numerically that both equilibria exhibit a supercritical Hopf bifurcation simultaneously. Hence, for parameter values $F$ below the corresponding Hopf bifurcation value $F_{\Hf}'$ the two equilibria $\xi^1_j$, $j=0,1$, are unstable and two stable periodic orbits coexist; see Figure~\ref{fig:Lz96-PFStructureFig1PF}.
  \item[Case 3: two pitchfork bifurcations] For $n=4k$, $k\in \Nat$, Theorems~\ref{thm:PFBif} and~\ref{thm:PF2Bif} guarantee the occurrence of two pitchfork bifurcations subsequently. By numerical continuation we observed that the resulting four stable equilibria exhibit supercritical Hopf bifurcations simultaneously \cite{Kekem17c1}. Thus, in this case four stable periodic orbits coexist for parameter values $F<F_{\Hf}''$; see Figure~\ref{fig:Lz96-PFStructureFig4m}.
\end{description}

\begin{figure}[ht!]
  \centering
   \makebox[\textwidth][c]{\includegraphics{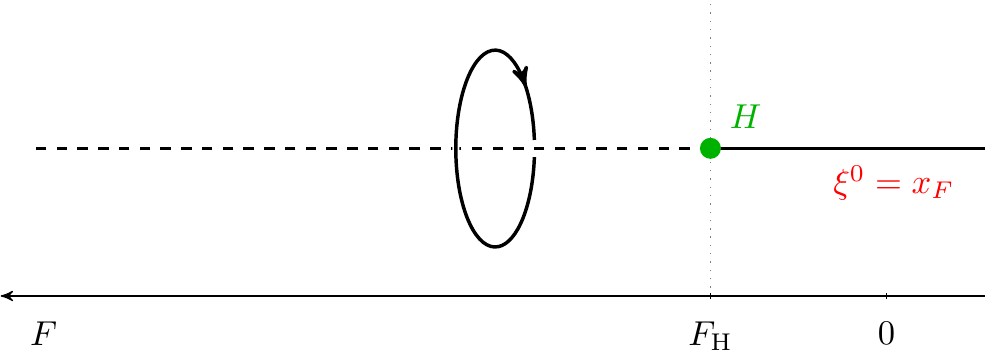}}\\
   \caption{Schematic representation of the attractors for negative $F$ in an $n$-dimensional Lorenz-96 model with odd $n>3$, so without any pitchfork bifurcation. The label $H$ stands for a (supercritical) Hopf bifurcation and occurs for $-0.894427\leq F_{\Hf}<-\tfrac{1}{2}$. The only equilibrium is given by $\xi^0 \equiv x_F\in \Fix(G_n^1)$. A solid line represents a stable attractor; a dashed line represents an unstable one.}\label{fig:Lz96-PFStructureFigOdd}
\end{figure}
\begin{figure}[ht!]
  \centering
   \makebox[\textwidth][c]{\includegraphics{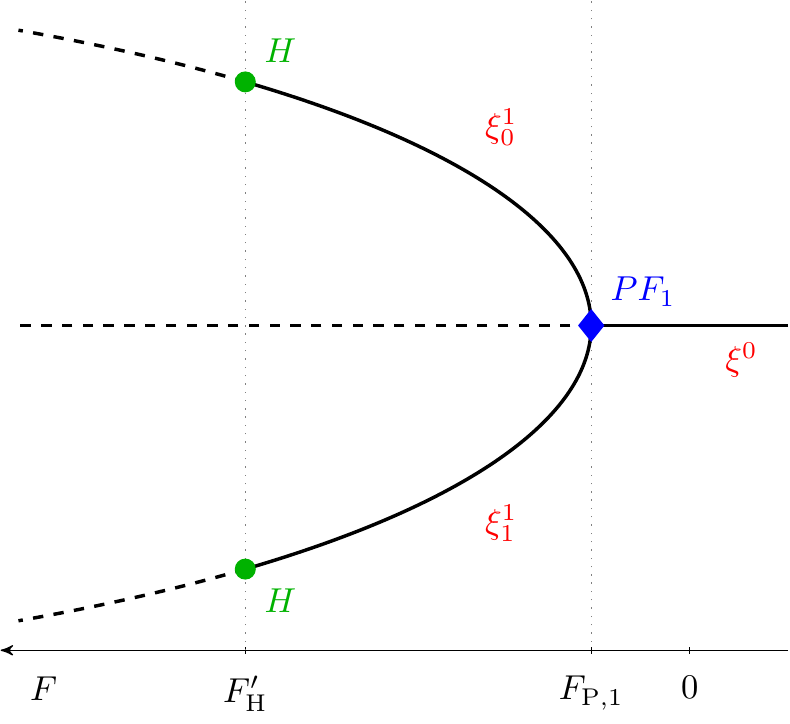}}\\
   \caption{Schematic bifurcation diagram of a $2^1 p$-dimensional Lorenz-96 model with $p>1$ odd and for negative $F$. The label $PF_1$ denotes the only (supercritical) pitchfork bifurcation with bifurcation value $F_{\PF,1}=-\tfrac{1}{2}$; $H$ stands for a (supercritical) Hopf bifurcation with bifurcation value $-3.5\leq F_{\Hf}'\leq -3$, depending on $n$. The equilibria are $\xi^0 \equiv x_F\in V^0$ and $\xi^1_j\in V^1$, $j=0,1$ given by equation~\eqref{eq:eqxi1}. A solid line represents a stable equilibrium; a dashed line represents an unstable one.}\label{fig:Lz96-PFStructureFig1PF}
\end{figure}
\begin{figure}[ht!]
  \centering
   \makebox[\textwidth][c]{\includegraphics{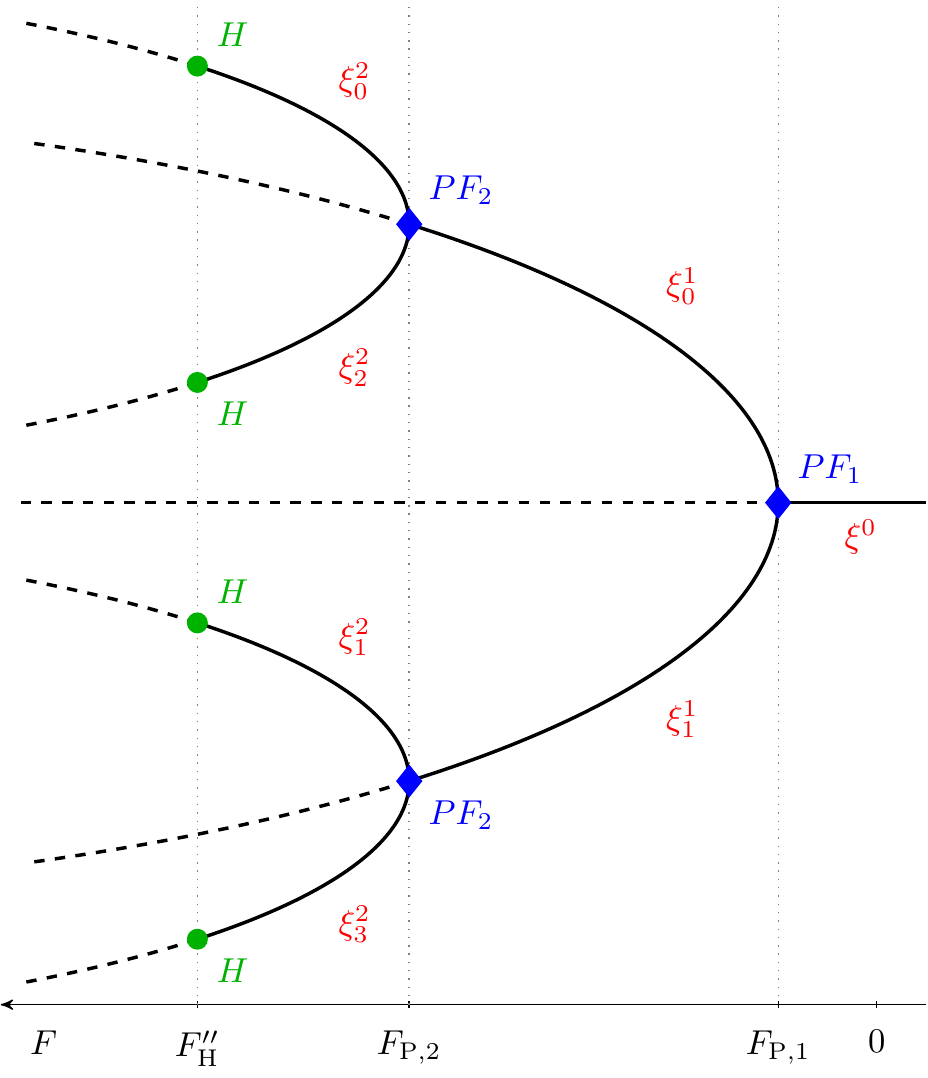}}\\
   \caption{Schematic bifurcation diagram of a $2^2 p$-dimensional Lorenz-96 model with $p>1$ odd and for negative $F$. The label $PF_1$, resp.~$PF_2$, denotes the first, resp.~second, (supercritical) pitchfork bifurcation with bifurcation value $F_{\PF,1}=-\tfrac{1}{2}$, resp.~$F_{\PF,2}=-3$; $H$ stands for a (supercritical) Hopf bifurcation with bifurcation value $-3.9< F_{\Hf}''< -3.5$, depending on $n$. The equilibria are given by $\xi^0 = x_F\in \Fix(G_n^1)$ and by equation~\eqref{eq:eqxi1} for $\xi^1_j\in \Fix(G_n^2)$, $j=0,1$, while $\xi^2_j\in \Fix(G_n^4)$, $j=0,\ldots,3$. A solid line represents a stable equilibrium; a dashed line represents an unstable one.}\label{fig:Lz96-PFStructureFig4m}
\end{figure}

\paragraph{Symmetry for periodic solutions}Recall from section~\ref{sec:symmetries} that
whenever $x(t)$ is a solution of the Lorenz-96 model, then $\gamma_n^j x(t)$ is a solution as well for any $1\leq j\leq n$. This also holds for a periodic solution $P(t)$ of system~\eqref{eq:Lorenz96}. The orbits of $P(t)$ and $\gamma_n^j P(t)$ are either identical or disjoint by uniqueness of solutions. In the first case both orbits differ at most by a phase shift in time; in the second case we obtain a new periodic solution $\gamma_n^j P(t)$ but whose spatiotemporal properties (i.e.\ the period and wave number) are the same as that of $P(t)$ \cite{Golubitsky88}.

In the Lorenz-96 model we observed numerically that the two or four periodic orbits, generated through the Hopf bifurcations after one or two pitchfork bifurcations, are indeed $\gamma_n$-conjugate to each other. Because they all emerge from a different equilibrium, their orbits must be disjoint, but they share the same spatiotemporal properties. Hence, due to the symmetry the periodic orbits for $F<0$ are related to each other by conjugacy as follows: Hopf bifurcation for negative $F$:
\begin{description}
  \item[Case 1 ($n$ odd)] There is only one periodic orbit $P(t)$ that satisfies $\gamma_n P(t) = P(t+jT/n)$, where $T$ is the period and $1\leq j < n$; i.e.~applying $\gamma_n$ results in a phase shift proportional to $T/n$ such that after $n$ iterations we retrieve the orbit without phase shift.
  \item[Case 2 ($n=4k+2$, $k\in\Nat$)] The two disjoint periodic orbits are $\gamma_n$-conjugate. Applying $\gamma_n$ twice returns the original periodic orbit but with a phase shift equal to $2jT/n$, where $1\leq j <n$. See Figure~\ref{fig:Lorenz96n6LCtimeF-36} for an example of the smallest dimension, $n=6$.
  \item[Case 3 ($n=4k$)] Four different periodic orbits exist of which three can be obtained from one by applying $\gamma_n$ subsequently one, two or three times as in dimension 4. Moreover, when we apply $\gamma_n$ four times, then the original periodic orbit reappears with a phase shift equal to $4jT/n$, $1\leq j < n$. See Figure~\ref{fig:Lorenz96n4LCtimeF-40} for an example of the smallest dimension, $n=4$.
\end{description}
More details about the spatiotemporal properties of the periodic attractors after the Hopf bifurcation in each of the cases listed above can be found in \cite{Kekem17c1}.

\begin{figure}[ht!]
  \centering
  \includegraphics[width=0.49\textwidth]{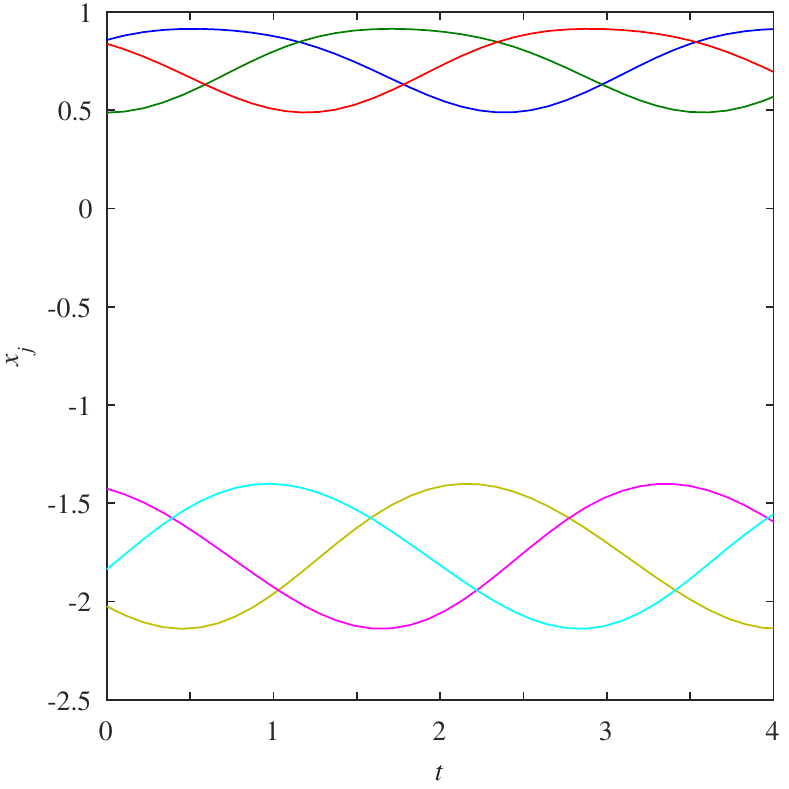}
  \includegraphics[width=0.49\textwidth]{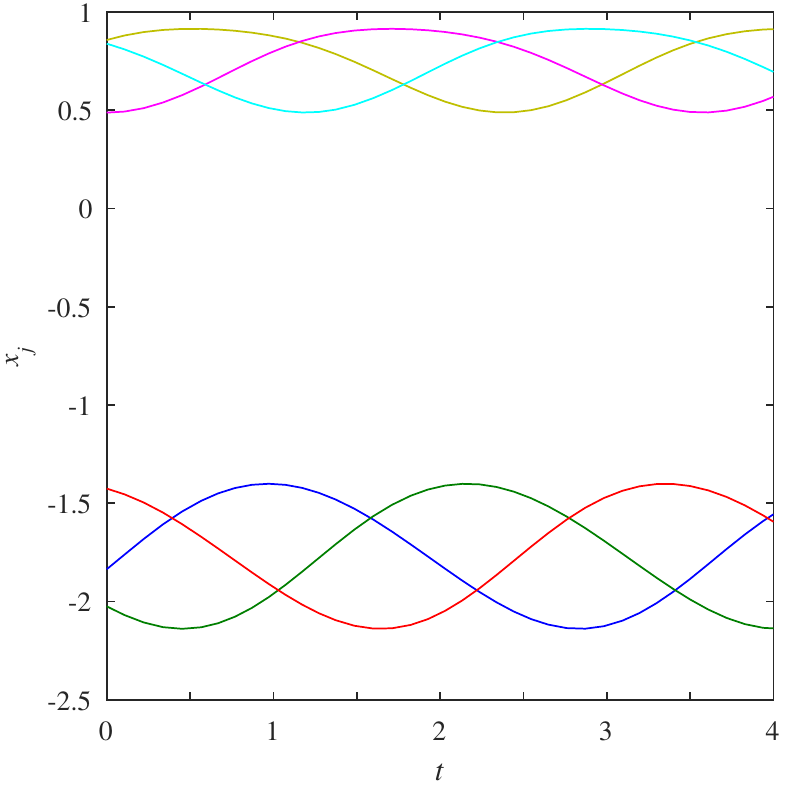}\\
  \caption{Time series of all coordinates $x_j$ of the two different periodic attractors for $n=6$ and $F=-3.6$, i.e.\ after the Hopf bifurcation following the first and only pitchfork bifurcation. The coordinates $x_j$, with $j=1,\ldots, 6$ are coloured blue, light-blue, red, purple, dark-green and yellow-green, respectively. The similarities between both periodic attractors are clear. A comparison of their coordinates shows that those of the right figure are shifted one place to the left with respect to the left one, which implies that the periodic orbits are $\gamma_6$-conjugate.}\label{fig:Lorenz96n6LCtimeF-36}
\end{figure}
\begin{figure}[ht!]
  \centering
  \includegraphics[width=0.49\textwidth]{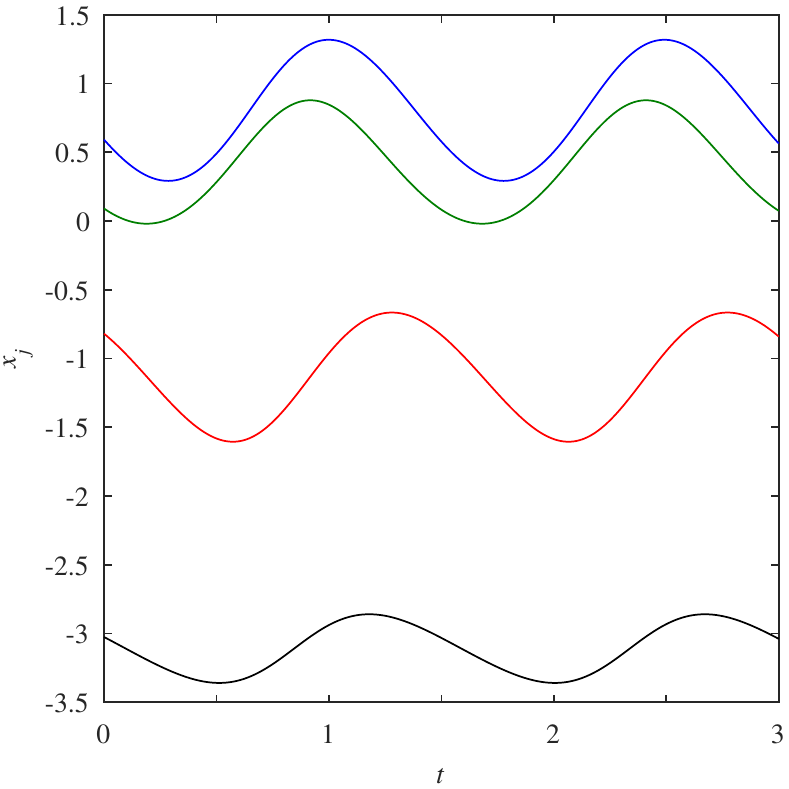}
  \includegraphics[width=0.49\textwidth]{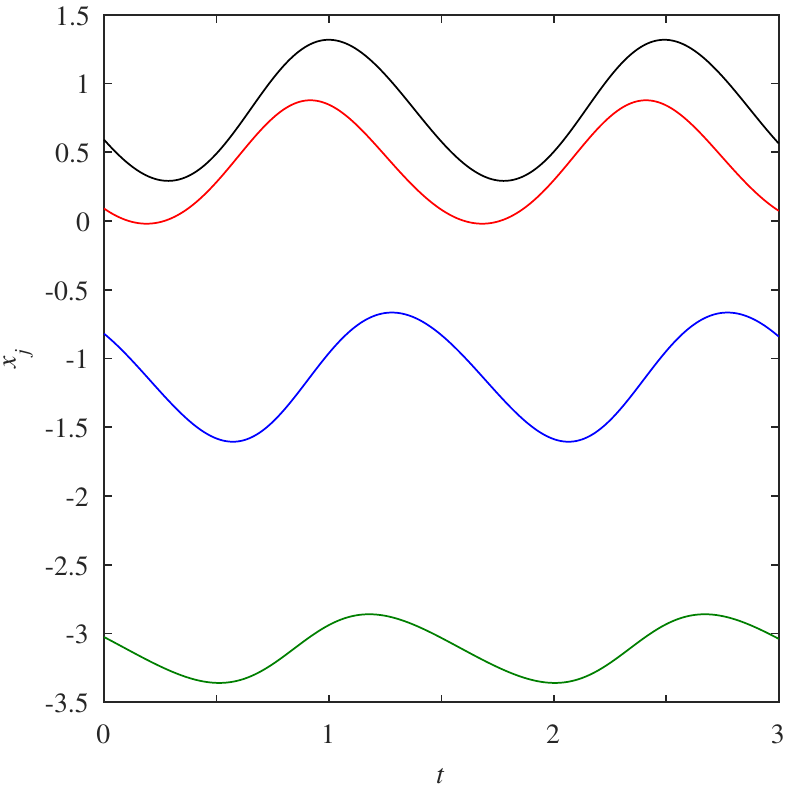}\\
  \caption{Time series of all coordinates $x_j$ of two periodic attractors for $n=4$ and $F=-4.0$, i.e.\ after the Hopf bifurcation following the second pitchfork bifurcation. The coordinates $x_j$, with $j=1,2,3,4$ are coloured blue, red, green and black, respectively. Observe that the periodic orbits are disjoint and $\gamma_4$-conjugate to each other, which means that they originate from the branches $\xi^2_k$ (left) and $\xi^2_{k+1}$ (right) with $k\in \{0,1,2,3\}$. The disjoint periodic orbits from the two other branches, $\xi^2_{k+2}$ and $\xi^2_{k+3}$, are obtained similarly, i.e.~by applying $\gamma_4$ two and three times to the periodic orbit in the left figure, in agreement with the analytical results (section~\ref{sec:PF2}).}\label{fig:Lorenz96n4LCtimeF-40}
\end{figure}

In order to check the symmetry of these periodic orbits, we perform the following numerical experiment. For a given dimension $n$ we follow the stable attractor for increasing or decreasing~$F$. We fix the value of the parameter $F$ and integrate the system long enough to obtain an attractor. After that, we check for repetition of the coordinates of the attractor. The number of different coordinates is then the dimension of the invariant subspace that contains the stable attractor. Finally, we raise or lower $F$ with a small step.

Using this method, we observe that, in general, the periodic orbits do not belong to any fixed-point subspace other than $\Fix(\Id_n)=\Real^n$, for dimensions up to 100. This might be due to the fact that the Hopf bifurcation values $F_{\Hf}'$ and $F_{\Hf}''$ are different for each dimension \cite{Kekem17c1}, which leads to different periodic orbits that do not inherit their properties from a lower dimension. However, in dimensions that are multiples of 6 we observe a tendency for periodic attractors in $\Fix(G^6_n)$ to become stable after a while; see Figure~\ref{fig:Symmn6}. This is observed in both dimensions of the form $n=4k$ and $n=4k+2$, so it could be the case that (even in dimensions $n=4k$) this symmetric attractor originates (via a Hopf bifurcation) from the equilibria directly after the first pitchfork bifurcation.
\begin{figure}[ht!]
  \includegraphics[width=1.02\textwidth]{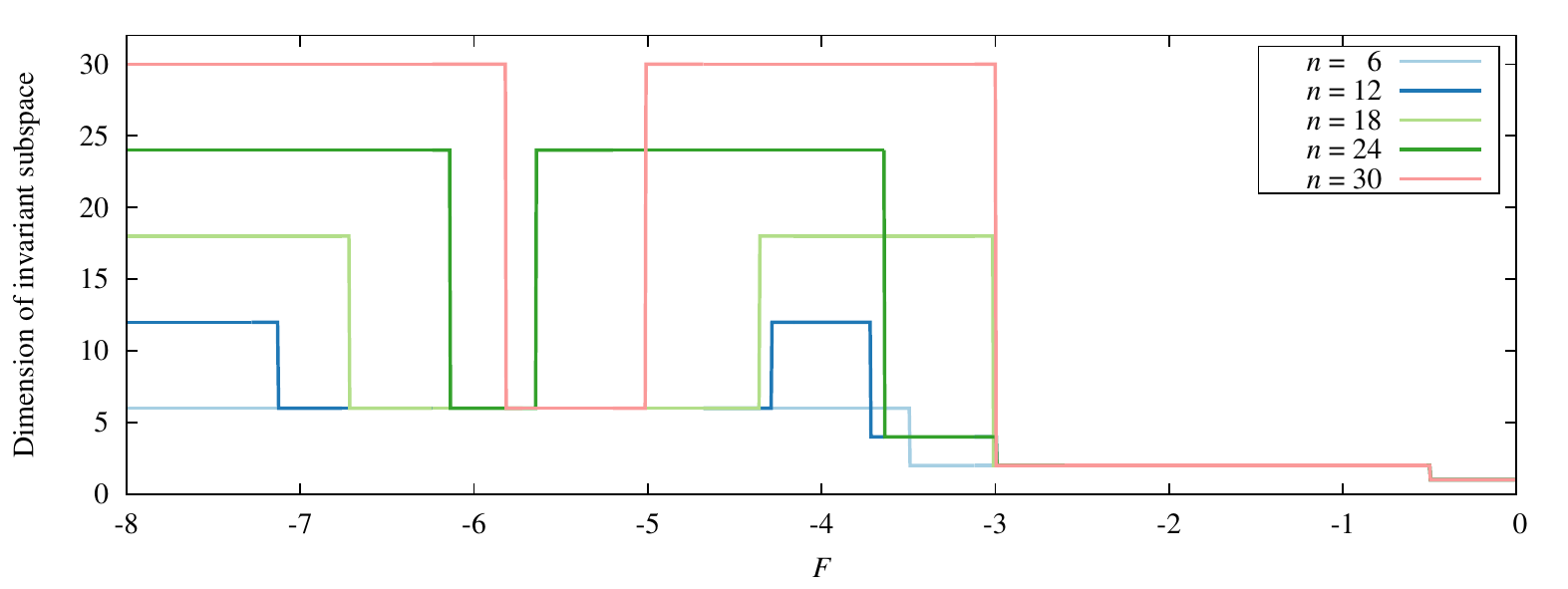}
  \caption{Plot of the dimension $m$ of the invariant subspace $\Fix(G^m_n)$ that contains the global attractor for various dimensions $n=6k$, $k=1,\ldots,5$, and negative $F$ (see text for the description of the method). In any dimension, after one or two pitchfork bifurcations a periodic orbit is generated with no symmetry (i.e.~contained only in $\Fix(G_n^n)$). For slightly smaller $F$, a symmetric attractor gains stability, which is any case contained in $\Fix(G_n^6)$.}\label{fig:Symmn6}
\end{figure}

\paragraph{Positive $F$}The trivial equilibrium $x_F$ is of the form $x^1 \in \Fix(G^1_n)$ and exists in any dimensions and for all $F\in\Real$. However, for positive forcing the first bifurcation for this equilibrium is not induced by symmetry, but it is either a supercritical Hopf bifurcation or a double-Hopf bifurcation, as we have shown in \cite{Kekem17a1}. This bifurcation happens at $F_{\Hf} (l_1,n) := 1/(\cos \tfrac{2\pi l_1}{n} - \cos \tfrac{4\pi l_1}{n})$, where $l_1$ denotes the index of the first eigenpair~\ref{eq:Lzevconj} crossing the imaginary axis, and results in one or more stable periodic orbits. Note that the index $l_1$ (which also represents the wave number of the periodic orbit \cite{Kekem17a1}) varies with the dimension. This results in a lot of different periodic orbits that can have various or no symmetry and their own route to chaos. Below, we will give a condition for which a periodic orbit has symmetry.

We investigated a few particular cases where it is observed that the periodic orbit is symmetric, using the same numerical experiment as above. For instance, for dimensions $n=5k$, $k=1,\ldots, 10$ a pattern of attractors is observed that are all invariant under $\gamma_n^5$ \cite{Kekem17a1}, which implies that they are contained in $\Fix(G^5_n)$ and inherit their properties partly from the attractor of $n=5$. This is confirmed by the plots in Figure~\ref{fig:Symmn5}, that shows the symmetry of the periodic orbits for dimensions $n=5k$, $k=1, \ldots, 12$. It can be seen that for $n=55$ and $60$ a symmetric attractor in $\Fix(G^5_n)$ becomes stable after a non-symmetric attractor has disappeared.
\begin{figure}[ht!]
  \includegraphics[width=1.02\textwidth]{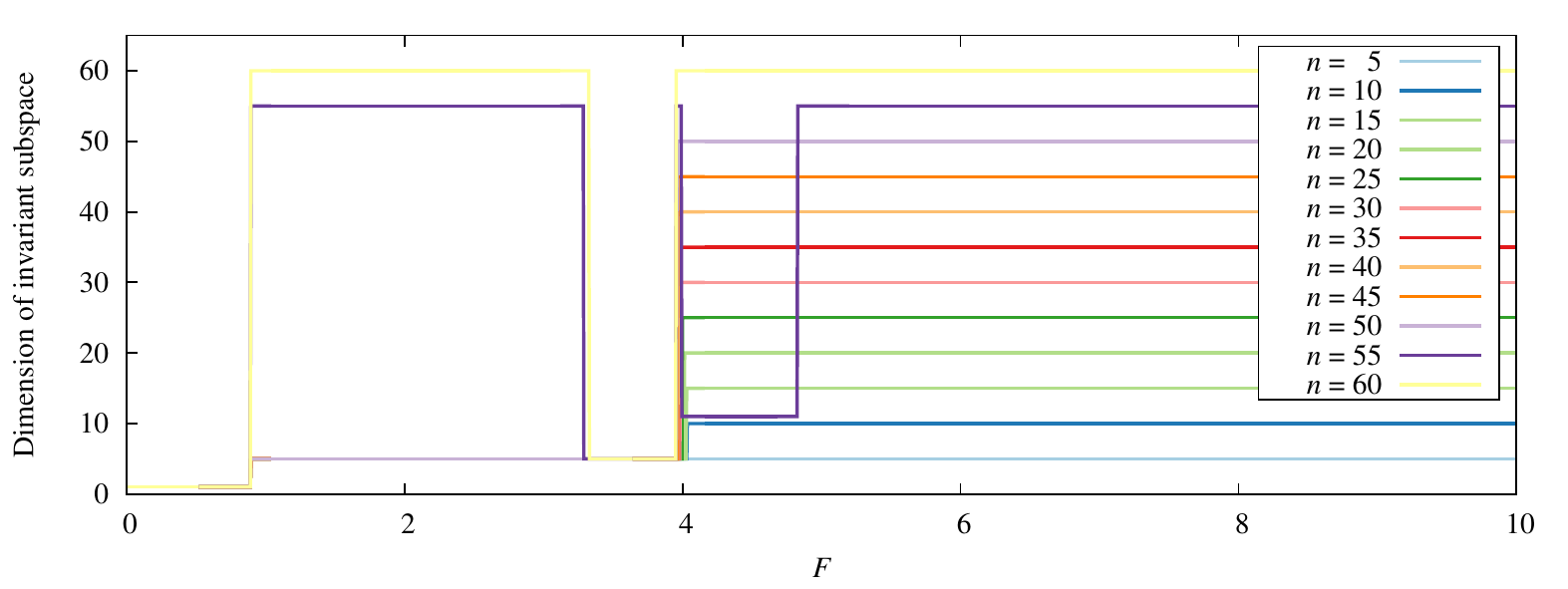}
  \caption{As Figure~\ref{fig:Symmn6}, but with dimensions $n=5k$, $k=1,\ldots,12$, and positive $F$. In any dimension up to $n=50$ an attractor is generated through a Hopf bifurcation which is contained in $\Fix(G^5_n)$. For $n=55$ and $60$ first an attractor without symmetry dominates, but for some larger values of $F$ an attractor in $\Fix(G^5_n)$ becomes globally stable again.}\label{fig:Symmn5}
\end{figure}

Similarly, in e.g.~$n=8$, resp.\ $n=12$ periodic orbits are observed with wave number $l=2$, resp.\ $l=2$ and 3 that are contained in $\Fix(G_8^4)$, resp.\ $\Fix(G_{12}^6)$ and $\Fix(G_{12}^4)$ (see Figure~\ref{fig:Symmn4}). In the same figure we show that for $n=28$ an attractor, with wave number $l=6$, exists that is contained in $\Fix(G_{28}^{14})$.
\begin{figure}[ht!]
  \includegraphics[width=1.02\textwidth]{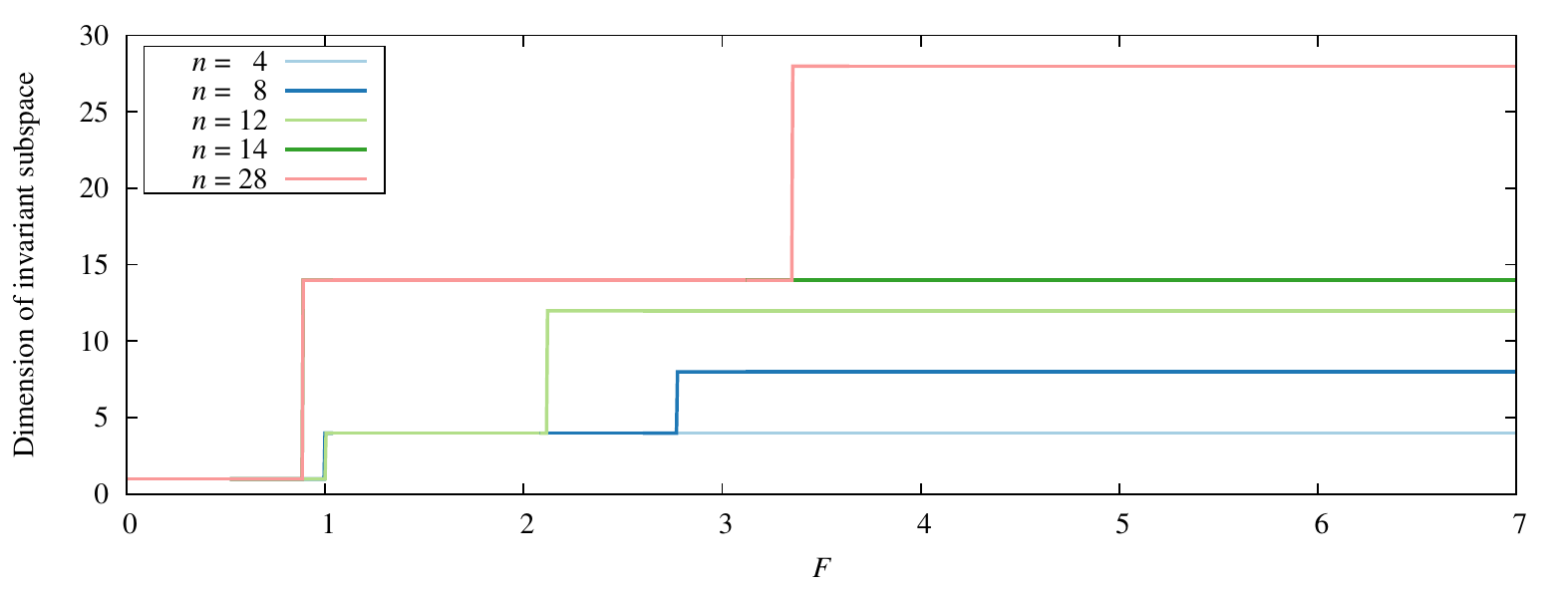}
  \caption{As Figure~\ref{fig:Symmn6}, but with various dimensions and positive $F$. For $n=8$ and $12$ the global attractor corresponds to the one for $n=4$. A similar phenomenon occurs for dimensions $n=14$ and $28$.}\label{fig:Symmn4}
\end{figure}

We therefore conjecture that when the spatial wave number $l$ of a periodic orbit $P(t)$ and the dimension $n$ satisfy $\gcd(l,n)=g>1$, then $P(t)\in\Fix(G_n^{n/g})$, i.e.\ the periodic orbit generated through the first Hopf bifurcation is symmetric. This phenomenon can be explained by the fact that such a wave splits into $g$ parts, where each part constitutes a wave with wave number $l/g$ that corresponds to the wave in dimension $n/g$. Note that this also includes the case where $\gcd(l,n)=l$, which is mentioned in \cite{Lorenz06b}. Periodic orbits with such a feature can arise in many dimensions, even if they are unstable as they emerge from a later Hopf bifurcation of the trivial equilibrium. A further discussion of this phenomenon lies beyond the scope of this article, but will be investigated in forthcoming work.

\subsection{Multiple pitchfork bifurcations}\label{sec:PFcascade}
In section~\ref{sec:AnalyticalResults} we have proven that it is possible to have two pitchfork bifurcation after each other, namely, when $n$ is a multiple of 4. Numerically, we observe that there can be even more subsequent pitchfork bifurcations after these two bifurcations. Even though these additional bifurcations happen after the Hopf bifurcations of Case 3 in the previous subsection (and therefore they occur for unstable equilibria and generate unstable equilibria), they can entail large groups of symmetries, an exponentially increasing number of equilibria and they show a beautiful structure. Since this possibly influences the dynamical structure for smaller $F$, we will discuss here the appearance of multiple pitchfork bifurcations and explain their presence using the exposition of symmetry from section~\ref{sec:AnalyticalResults}.

In the following exposition we write the dimension uniquely as $n = 2^q p$, with $q \in \Nat \cup\{0\}$ arbitrary and $p$ odd. One should bear in mind that the cases $q=1$ and $q=2$ are completely covered by the results proven in sections~\ref{sec:PF1} and~\ref{sec:PF2}. Consequently, we assume $q\geq3$ in the following, which enables the occurrence of more than two subsequent pitchfork bifurcations. and are complementary to the analytical results. Let us start with some notation that anticipates the results later on.

\paragraph{Notation}First of all, we call the pitchfork bifurcation which is the $l$-th in the row the $l$-\emph{th pitchfork bifurcation} and denote its bifurcation value as $F_{\PF,l}$. Clearly, $F_{\PF,l} < F_{\PF,l-1}$ and by definition the $l$-th pitchfork bifurcation occurs for equilibria generated through the $(l-1)$-th bifurcation. In the previous sections we have already used this nomenclature for the cases $l=1, 2$.

Furthermore, we will introduce some notation that anticipates the results later on. The groups~\eqref{eq:subgroup} and invariant subspaces~\eqref{eq:fixedpointsubspaceGmn} with $m=2^l$ such that $0\leq l\leq q$ are of particular importance in the description of the symmetry and related pitchfork bifurcations. Therefore, we define the special invariant subspace $\Fix(G_n^{2^l})\subset \Real^n$ of system~\eqref{eq:Lorenz96} as
  \begin{equation}\label{eq:invmanVl}
    V^l := \Fix(G_n^{2^l}) = \{x\in\Real^n: x_{j+2^l} = x_j\ \textrm{for all}\ 0\leq j \leq n-1\},
  \end{equation}
  where $0\leq l\leq q$ and the index of $x$ has to be taken modulo $n$.
Note that these invariant manifolds also played a crucial role in the proofs of Theorems~\ref{thm:PFBif} and~\ref{thm:PF2Bif}.

By the discussion in section~\ref{sec:symmetries} it is easy to see that each invariant manifold $V^l$ contains all of its `predecessors':
\begin{equation*}
  V^{l'}\subset V^l, \qquad 0\leq l' \leq l.
\end{equation*}
Also, by the definition of $V^l$, Proposition~\ref{prop:invmanifold} immediately implies the following result:
\begin{cor}\label{cor:invmanifoldVl}
    Let $n=2^q p$ and $0\leq l \leq q$. Then the dynamics of the $n$-dimensional Lorenz-96 model restricted to the invariant manifold $V^l$ is topologically equivalent to the Lorenz-96 model of dimension $2^l$.
\end{cor}

Furthermore, inspired by Remark~\ref{rmk:Kuznetsov} and equation~\eqref{eq:eqxi1} we also define
\begin{equation}\label{eq:eqxilj}
  \xi^l_j \in V^l, \qquad 0 \leq j \leq 2^l-1,
\end{equation}
to be the equilibria generated by the $l$-th pitchfork bifurcation, which have the same symmetry as equilibria of the form $x^{2^l}$. Similarly, let $\xi^l$ be the collection of all equilibria $\xi^l_j$,
\begin{equation*}
  \xi^l := \{ \xi^l_j \in V^l,\ 0 \leq j \leq 2^l-1\} \subset V^l,
\end{equation*}
which turns out to contain all equilibria that share the same properties. Accordingly, for $l=1$ we have $\xi^1 = \{\xi^1_0,\xi^1_1\} \subset V^1$ as defined in equation~\eqref{eq:eqxi1}. Likewise, we can define $\xi^0 \equiv x_F \in V^0$, for convenience.

\paragraph{Numerical observations}In Table~\ref{tab:ListnumberPF} we list the numbers of successive pitchfork bifurcations that are observed for specific even dimensions as well as the total number of equilibria generated through these bifurcations including the trivial equilibrium~$x_F$ (right column). The number of pitchfork bifurcations for a specific dimension $n=2^qp$, as above, turns out to be precisely the exponent $q$. Accordingly, we assume in the following that $0\leq l \leq q$, which coincides with the restriction for $V^l$ in equation~\eqref{eq:invmanVl}.

\begin{table}[ht!]\centering
    \caption{The number of successive pitchfork bifurcations and the corresponding total number of (possibly unstable) equilibria after the last pitchfork bifurcation as observed in selected even dimensions.}
    \label{tab:ListnumberPF}
    \begin{tabular}{rrr}
        \hline
        $n$ & \#PF's & \#equilibria \\
        \hline
        2 & 1 & 3 \\
        4 & 2 & 7 \\
        6 & 1 & 3 \\
        8 & 3 & 15 \\
        10 & 1 & 3 \\
        12 & 2 & 7 \\
        14 & 1 & 3 \\
        16 & 4 & 31 \\
        20 & 2 & 7 \\
        24 & 3 & 15 \\
        32 & 5 & 63 \\
        36 & 2 & 7 \\
        64 & 6 & 127 \\
        128 & 7 & 255 \\
        256 & 8 & 511 \\
        512 & 9 & 1023 \\
        \hline
    \end{tabular}
\end{table}

Besides, the bifurcation values $F_{\PF,l}$ are independent of $n$ for all $l$. These fixed values $F_{\PF,l}$ are listed in Table~\ref{tab:ListFPj} for $l\leq9$ and are obtained by numerical continuation in the dimensions $n=2^l$ using the software packages \textsc{Auto-07p} \cite{Doedel12} and \textsc{MatCont} \cite{Dhooge11}. In addition, the $l$-th pitchfork bifurcation occurs for \emph{all} equilibria $\xi^{l-1}_j(F)$ at exactly the same bifurcation value $F_{\PF,l}$. So, when we speak about `the $l$-th pitchfork bifurcation' there are actually $2^{l-1}$ simultaneous pitchfork bifurcations of conjugate equilibria, generating in total $2^l$ new equilibria.

Even more, we observed that all these new equilibria have the same entries in the same order but shifted, which justifies our notation of the equilibria~\eqref{eq:eqxilj}. Therefore, the equilibria $\xi^l_j\in V^l$ satisfy in general
\begin{equation}\label{eq:conjeq}
  \gamma_n^k \xi^l_j = \xi^l_{j+k}, \qquad \text{for all}\ 0 \leq j, k \leq 2^l-1,
\end{equation}
where the lower index of $\xi$ should be taken modulo $2^l$. As described in section~\ref{sec:symmetries}, all these conjugate solutions have the same properties and therefore it suffices to study only one copy of them, say $\xi^l_0$. We will often just refer to the set $\xi^l$ (so, without index) when we describe their common properties.

\begin{table}[ht!]\centering
    \caption{List of bifurcation values $F_{\PF,l}$ for the $l$-th pitchfork bifurcation, which are known up to $l=9$ and that are independent of the dimension $n$. The two right columns give the distances between the successive pitchfork bifurcations and their ratios $r_l = (F_{\PF,l-1}-F_{\PF,l-2})/(F_{\PF,l}-F_{\PF,l-1})$.}
    \label{tab:ListFPj}
    \begin{tabular}{rrrr}
        \hline
        $l$ & $F_{\PF,l}$ & Distance to $F_{\PF,l-1}$ & $r_l$\\
        \hline
        1 & $-0.5$       & --        & -- \\
        2 & $-3$         & 2.5       & -- \\
        3 & $-6.6$       & 3.6       & 0.694444 \\
        4 & $-8.0107123$ & 1.41071   & 2.55190 \\
        5 & $-8.4360408$ & 0.425329  & 3.31676 \\
        6 & $-8.5275625$ & 0.0915217 & 4.64730 \\
        7 & $-8.5474569$ & 0.0198944 & 4.60037 \\
        8 & $-8.5517234$ & $4.2665\times10^{-3}$ & 4.66289 \\
        9 & $-8.5526377$ & $9.143\times10^{-4}$  & 4.66681 \\
        \hline
    \end{tabular}
\end{table}
Table~\ref{tab:ListFPj} also shows that the distance between successive pitchfork bifurcations decreases as $l$ increases. The values of their ratios $r_l$ suggest that
\begin{equation*}
  \lim_{l\rightarrow\infty} r_l = \lim_{l\rightarrow\infty} \frac{F_{\PF,l-1}-F_{\PF,l-2}}{F_{\PF,l}-F_{\PF,l-1}} = \delta,
\end{equation*}
where $\delta \approx 4.66920$ is Feigenbaum's constant. Therefore, the $q$-th and last pitchfork bifurcation of a specific dimension $n$ will be expected for the bifurcation value $F_{\PF,q} \geq F_{\PF,\infty} \approx -8.55289$.

\paragraph{Visualisation of structure}The structure of pitchfork bifurcations and equilibria that we observed by numerical analysis is summarised in Figure~\ref{fig:Lz96-PFStructureFullFig}, which we will now explain. The figure presents a schematic view for the case $n=2^qp$ with $q=4$ and gives an indication for the bifurcation structure for general $q\geq3$. First of all, the horizontal line in the middle represents the trivial equilibrium $\xi^0 = x_F$ which is stable for $F> F_{\PF,1}$. At the point $PF_1$ we see that two stable equilibria $\xi^1_{0,1}$ emerge, while $\xi^0$ becomes unstable: the first supercritical pitchfork bifurcation.

\begin{figure}[p]
  \centering
  \vspace*{-8em}
   \makebox[\textwidth][c]{\includegraphics{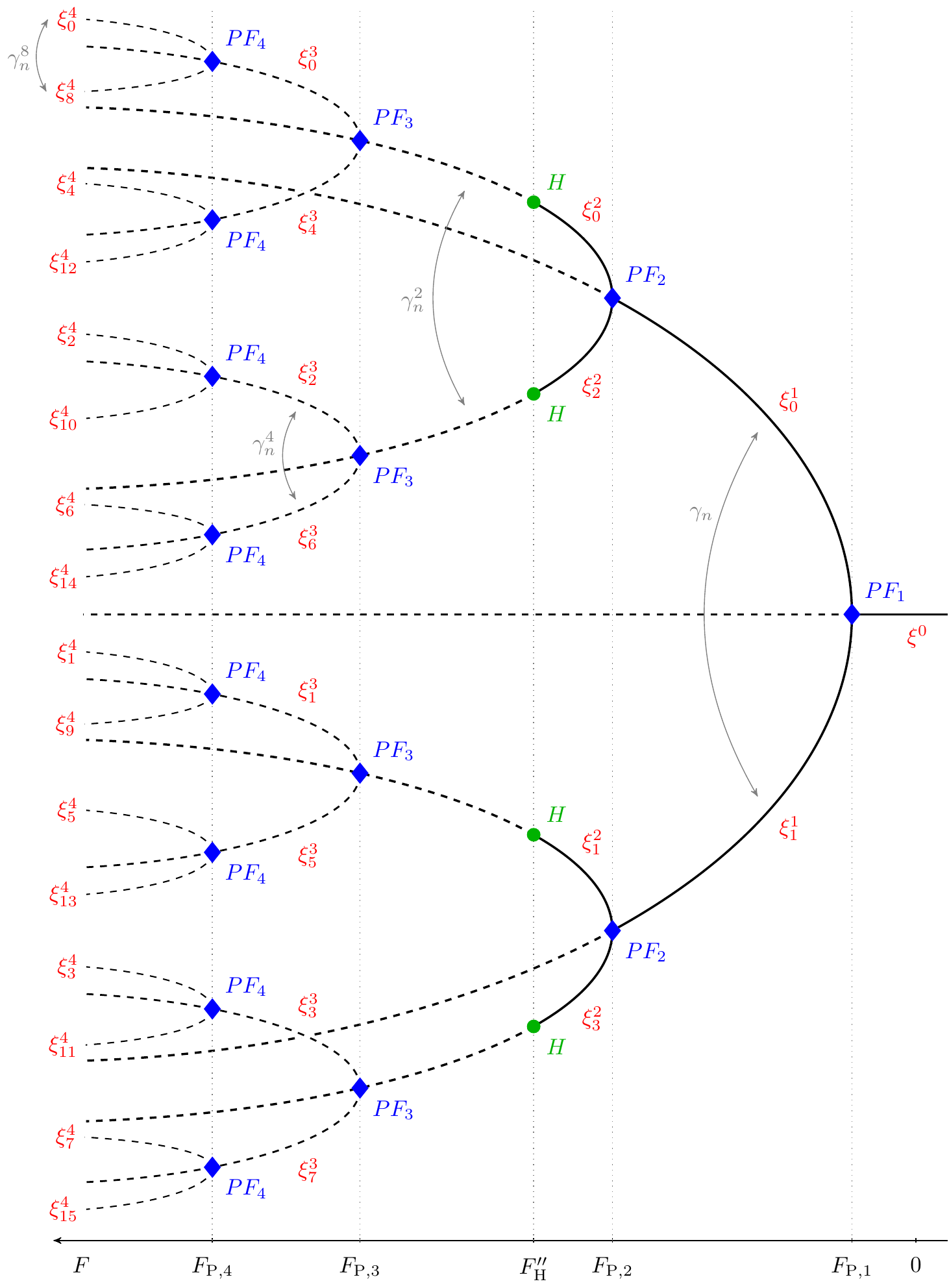}}\\
   \caption{Schematic bifurcation diagram of a $2^qp$-dimensional Lorenz-96 model for negative $F$ with $q = 4$ subsequent pitchfork bifurcations. The label $PF_l$, $1\leq l\leq q$, denotes the $l$-th (supercritical) pitchfork bifurcation with bifurcation value $F_{\PF,l}$ as in Table~\ref{tab:ListFPj}; $H$ stands for a (supercritical) Hopf bifurcation with bifurcation value $-3.9<F_{\Hf}''<-3.5$. Each branch of equilibria is labelled with $\xi^l_j$ according to equation~\eqref{eq:eqxilj}, where $l$ indicates that the branch is generated by the $l$-th pitchfork bifurcation and contained in $V^l$ and $j$ denotes how often we have to apply $\gamma_n$ to $\xi^l_0$ to obtain this branch, as in equation~\eqref{eq:conjeq}. A solid line represents a stable equilibrium; a dashed line represents an unstable one. The arrows in gray indicate the relation between the mutual branches. Similar diagrams can be obtained for any $q\geq3$.}\label{fig:Lz96-PFStructureFullFig}
\end{figure}

Secondly, both equilibria $\xi^1_{0,1}$ exhibit a pitchfork bifurcation $PF_2$ at $F_{\PF,2}$. In both cases a pair of stable, $\gamma_n^2$-conjugate equilibria appear, i.e.~$\xi^2_2=\gamma_n^2 \xi^2_0$ and $\xi^2_3=\gamma_n^2 \xi^2_1$. Moreover, by formula~\eqref{eq:conjeq} these pairs are also $\gamma_n$-conjugate to each other, which means that we can switch between the branches originating from $\xi^1_0$ and those from $\xi^1_1$ by applying $\gamma_n$.

Next, all four equilibria from $\xi^2$ exhibit a supercritical Hopf bifurcation, by Case 3 of section~\ref{sec:Hbif}. As a result, $\xi^2$ and all successive equilibria $\xi^l$, $2<l\leq q$ are unstable for $F< F_{\Hf}''$. Thereafter, the third pitchfork bifurcation $PF_3$ occurs at $F_{\PF,3}$ and generates $2^3$ unstable and pairwise $\gamma_n^4$-conjugate equilibria $\xi^3$. Finally, the fourth pitchfork bifurcation generates the equilibria $\xi^4\subset V^4$. This completes the full structure with $2^5-1$ unstable equilibria.

\paragraph{Explanation by symmetry}The preceding phenomena can be explained using the concepts introduced in section~\ref{sec:symmetries}. In general, at a pitchfork bifurcation there is a breaking of the symmetry: before the bifurcation there exist an equilibrium $x_0$ satisfying $R x_0 = x_0$, where $R$ represents $\Integer_2$-symmetry, while after the bifurcation two additional equilibria $x_{1,2}$ appear that satisfy $R x_1 = x_2$ \cite{Kuznetsov04}. So, the new equilibria $x_{1,2}$ after the bifurcation have a lower order of symmetry than the bifurcating equilibrium $x_0$, as explained by Remark~\ref{rmk:Kuznetsov}. In terms of the invariant subspaces, this means that the smallest invariant subspace containing $x_{1,2}$ should be larger than the one containing $x_0$. More explicitly: if $x_0 \in \Fix(G_n^m)$, with $m\leq \tfrac{n}{2}$ minimized, then the two resulting equilibria $x_{1,2}$ are in $\Fix(G_n^{m'})$ with $m'=2m$ (due to $\Integer_2$-symmetry).

In section~\ref{sec:PF1} we have demonstrated that the equilibrium $\xi^0 = x_F \in V^0$ exhibits the first pitchfork bifurcation and that two stable equilibria $\xi^1 \subset V^1$ appear. The second pitchfork bifurcation occurs for both equilibria $\xi^1$ simultaneously and generates stable equilibria $\xi^2\subset \Fix(G^4_n) = V^2$, as shown in section~\ref{sec:PF2}. In general, assuming that $l\leq q$ and that the equilibria $\xi^{l-1}\subset V^{l-1}$ generated through the $(l-1)$-th pitchfork bifurcation again exhibit a pitchfork bifurcation, then the $l$-th pitchfork bifurcation generates $2^l$ new branches of equilibria $\xi^l_j(F)\in V^l$, $0\leq j\leq 2^l-1$, where $F<F_{\PF,l}$. Thus, the total number of equilibria for dimension $n$ generated by the $q$ pitchfork bifurcations (including the trivial equilibrium) is equal to $2^{q+1}-1$, which is confirmed by the right column of Table~\ref{tab:ListnumberPF}.

The observation that the $l$-th pitchfork bifurcation consists of $2^{l-1}$ simultaneous pitchfork bifurcations of conjugate equilibria can be explained by noting that all equilibria $\xi^{l-1}$ satisfy the relation~\eqref{eq:conjeq} and therefore share the same properties and, in particular, the same eigenvalues. The fact that the bifurcation values $F_{\PF,l}$ do not depend on the dimension is a direct consequence of Proposition~\ref{prop:invmanifold}.

In particular, the $q$-th pitchfork bifurcation generates equilibria $\xi^q \subset V^q = \Fix(G^{2^q}_n)$. Consequently, there cannot be more than $q$ subsequent pitchfork bifurcations because this requires the resulting equilibria to be in $\Fix(G^{2^{q+1}}_n)$, which does not exist. Hence, for any dimension $n=2^qp$ there can be at most $q$ pitchfork bifurcations.

Based on these numerical observations and their interpretation in terms of symmetry, the following conjecture seems plausible:
\begin{cnj}
  The number of subsequent pitchfork bifurcations in the Lorenz-96 model of dimension $n= 2^q p$, where $q\in\Nat\cup\{0\}$ and $p$ odd, is exactly equal to $q$.
\end{cnj}
In summary, the results in this section show that in each dimension $n=2^qp$ there are exactly $q$ pitchfork bifurcations for $F<0$ and the phenomenon fits well into the theoretical description given in section~\ref{sec:symmetries}.

\section{Conclusions and outlook}\label{sec:Conclusion}
In this investigation, the aim was to unravel the symmetrical nature of the Lorenz-96 model and to understand its dynamics better with this information, building on \cite{Kekem17a1}. The model is equivariant in any dimension with respect to a cyclic left shift $\gamma_n$ and the groups of symmetries give rise to invariant manifolds for each divisor $m$ of the dimension $n$. One of the major findings of this paper was that the invariant manifolds allow us to extrapolate results that are proven for a certain dimension $n$ to all multiples of $n$. These findings enhance our understanding of the Lorenz-96 model.

In the present study we exploited the symmetry mainly to study and explain the dynamics for negative parameter values, where symmetry turns out to play an important role. We have proven analytically the existence of one, resp.\ two, pitchfork bifurcations in dimension $n=2$, resp.\ $n=4$, each of which gives rise to $\gamma_n^{n/2}$-conjugate equilibria. Consequently, in any even dimension a pitchfork bifurcation takes place, with an additional subsequent pitchfork bifurcation when the dimension equals $n=4k$, $k\in\Nat$.

Numerical investigation shows another significant finding of this study: in any dimension $n$ the number of successive pitchfork bifurcations is exactly equal to $q$, where $q$ is the nonnegative integer such that the dimension $n$ is uniquely given by $n=2^qp$, with $p$ odd. However, to establish this result analytically is a nontrivial task, since the Jacobian is no longer circulant for the nontrivial equilibria that arise from each pitchfork bifurcation. Moreover, to prove other facts beyond the $l$-th pitchfork bifurcation will become increasingly difficult, since the lowest dimension needed is $n=2^l$ and thus increases exponentially with $l$. On the other hand, once we have found an equilibrium $\xi^l_j$ in a certain invariant subspace $V^l$, the relation~\eqref{eq:conjeq} guarantees that the $2^l-1$ other equilibria have the same properties. This finding together with the $\Integer_n$-symmetry may have important implications for the dynamics after the cascade of pitchfork bifurcations. Although the periodic orbit that emerges from the supercritical Hopf bifurcation is the stable attractor for $F<F_{\Hf}''$ (see section~\ref{sec:Hbif}), the cascade-like pitchfork bifurcations can have a big influence on the dynamics via the large number of generated equilibria. Such an influence has been observed in dimension $n=4$ for positive $F$, where 4 unstable and $\gamma_4$-conjugate equilibria give rise to a heteroclinic structure that causes the dynamics on the chaotic attractor to return to nearly periodic behaviour repeatedly, i.e.~the classical type 1 intermittency scenario \cite{Kekem17a1}.

The influence of the symmetry on the Lorenz-96 model for $F>0$ is less clear. We have shown in \cite{Kekem17a1} that the first bifurcations for the trivial equilibrium are always Hopf or Hopf-Hopf bifurcations. The emerging periodic orbits have symmetries under certain circumstances, namely when their wave number has a divisor in common with the dimension of the model. However, further work needs to be done to establish this conjecture. More symmetries might be found via other equilibria than the trivial one, but it is in general nontrivial to locate them.

Furthermore, a pattern of attractors for $n=5m$, $m=1,\ldots, 10$ is discussed in \cite{Kekem17a1}. This phenomenon can be explained by the invariant manifolds which allow us to extrapolate the results for low dimension to higher dimensions. However, as explained in Remark~\ref{rmk:invmanifold}, this method does not guarantee to give the complete bifurcation pattern and route to chaos for any multiple of the lowest possible dimension, but only the dynamics restricted to the corresponding invariant manifold. It is also possible that another bifurcation will take place before the phenomena extrapolated from low dimension and thus a different attractor gains stability. Such an event is indeed observed in the example of the pattern for the dimensions $n=5m$, where the pattern is interrupted at $m=11$.

Altogether, the results in this paper provide important insights into the symmetrical structure of the Lorenz-96 model. This also helps to understand the bigger dynamical structure and its travelling waves, partly described in \cite{Kekem17a1,Kekem17c1,Orrell03}. Further studies need to be carried out in order to unravel the bifurcations and routes to chaos of the stable attractors for negative $F$. An interesting question in this context is how the symmetry influences the dynamics for parameter values beyond the pitchfork bifurcations or for larger dimensions. The attractor for $n=4$ and large negative $F$ is studied in \cite{Lorenz84}, although without taking into account its potential symmetry. In particular, note that for $F<0$ the bifurcation patterns up to and including the Hopf bifurcation can be divided into three different cases, which might have consequences for the number of different routes to chaos. In our future research we will therefore include a further analysis of the bifurcation structure in combination with the symmetry of the model and explore the routes to chaos for $F<0$.

\DeclareRobustCommand{\VAN}[3]{#3}
\newpage\label{sec:Bibliography}
\bibliographystyle{Lz96Symmetry-DvKAES}
\bibliography{Lz96Symmetry-DvKAES}
\addcontentsline{toc}{section}{References}

\end{document}